\documentclass[12pt]{wart}
\usepackage{xspace,amssymb,amsfonts,euscript,eufrak,amsthm,amsmath}
\usepackage{graphicx,ifpdf}
\ifpdf \usepackage{epstopdf}
 \usepackage{pdfsync}\fi
\usepackage{palatino}
\usepackage{float}
\usepackage[usenames,dvipsnames,svgnames,table]{xcolor}

\usepackage{float}
\usepackage[all]{xy}
\usepackage{mathrsfs}
\usepackage[latin1]{inputenc}

 \title[Gentle introduction to Soergel bimodules I]
 {Gentle introduction to Soergel bimodules I:\\
  \large The basics}
\author{\hspace{1.38cm} \textsc{Nicolas Libedinsky }}

\input xy
\xyoption {all}

  \newcommand{\nc}{\newcommand}
  \newcommand{\renc}{\renewcommand}

\nc{\E}{\mathbb{E}}
\nc{\Q}{\mathbb{Q}}

\nc{\triright}{\stackrel{[1]}{\to}}

\def\to{\rightarrow}

\nc{\Br}{\mathcal{B}}
\nc{\id}{id}
\nc{\HotRR}{{}_R\mathcal{K}_R}
\nc{\HotR}{\mathcal{K}_R}
\nc{\excise}[1]{}
\nc{\defect}{\text{df}}
\nc{\h}[1]{b_{#1}}
\nc{\Z}{\mathbb{Z}}
\nc{\R}{\mathbb{R}}
\nc{\C}{\mathbb{C}}
\renc{\P}{\mathbb{P}}
\renc{\O}{\mathcal{O}}
\nc{\N}{\mathbb{N}}

\nc{\F}{\mathcal{F}}
\nc{\G}{\mathcal{G}}

\nc{\nilp}{\mathcal{N}}

\nc{\Ga}{\mathbb{G}_a} 
\nc{\Gm}{\mathbb{G}_m} 

\nc{\Loc}{\mathcal{L}}

\nc{\A}{\mathbb{A}} 

\nc{\IC}{\mathbf{IC}}
\nc{\D}{\mathbb{D}}

  \newtheorem{defi}{Definition}
  \newtheorem{thm}{Theorem}[section]
  \newtheorem{lem}[thm]{Lemma}
   \newtheorem{claim}[thm]{Claim}
  
  \newtheorem{prop}[thm]{Proposition}

  \newtheorem{nota}[thm]{Notation}
\newtheorem{ques}[thm]{Question}
\newtheorem{st}[thm]{Baby Stability Theorem}

\newtheorem{ef}[thm]{Easy Fact}

\newtheorem{sca}[thm]{Soergel's categorification Theorem}

  \theoremstyle{remark}
  \newtheorem{remark}{Remark}

\def\cA{\mathcal A}\def\cB{\mathcal B}
\def\cH{\mathcal H}

\def\II{\mathbb I}
\def\NN{\mathbb N}
\def\RR{\mathbb R}

\def\ZZ{\mathbb Z}

\def\del{\partial}

\begin{document}

\begin{abstract}
This paper is the first of a series of introductory papers on the fascinating world of Soergel bimodules. It is  combinatorial  in nature and should be accessible to a broad audience. The objective of this paper is to help the reader feel comfortable calculating with Soergel bimodules and to explain some of the important open problems in the field. The motivations, history and relations to other fields will be developed in subsequent papers of this series.

\end{abstract}
\maketitle
\section{Introduction}

\subsection{Declaration of intent}

As said in the abstract, this paper is an introduction to  Soergel bimodules. We give many examples and show explicit calculations with the Hecke algebra and the Hecke category (one of its incarnations being Soergel bimodules). Most of the  other Hecke categories (categorifications of the Hecke algebra)  such as category $\mathcal{O}$, Elias-Williamson diagrammatic category, Sheaves on moment graphs, 2-braid groups and Parity Sheaves over Schubert varieties, are left for follow-ups of this paper. The applications of this theory are also left for the follow-ups. 

Soergel bimodules were introduced  by Wolfgang Soergel \cite{So0} in the year 92', although many of the ideas were already present in his 90' paper \cite{So-1}. In those papers he explained its relations to representations of Lie groups. In the year 00' he proved \cite{So1}  a link between them (at that time he called them ``Special bimodules", and although they are quite special, apparently they are more Soergel than special) and representations of algebraic groups in positive characteristic that proved to be extremely deep. In his 07' paper  \cite{So3} he simplified many arguments and proved some new things. After this paper...

\textbf{WARNING: } The following section is just intended to impress the reader.

\subsection{...Bum! Explosion of the field}  Soergel bimodules were (and are) in the heart of an explosion of new discoveries in representation theory, algebraic combinatorics, algebraic geometry and knot theory.  We give a list of some results obtained using Soergel bimodules in the last five years.

\begin{enumerate}
\item \textbf{Algebraic groups:} A disproof of Lusztig's conjecture predicting the simple characters of reductive algebraic groups (1980).  This was probably the most important open conjecture in representation theory of Lie-type objects.  
\item \textbf{Lie algebras:} An algebraic proof of Kazhdan-Lusztig conjectures predicting the multiplicities of simple modules in Verma modules  (1979) for complex semi-simple Lie algebras. A geometric proof was given in the early 80's, but we had to wait 35 years to have an algebraic proof of an algebraic problem. 
\item \textbf{Symmetric groups: } A disproof of James conjecture predicting the characters of irreducible modular representations for the symmetric group (1990). 
\item  \textbf{Combinatorics: }A proof of the conjecture about the positivity of the coefficients of Kazhdan-Lusztig polynomials for any Coxeter system (1979). This was a major  open combinatorial problem. 
\item \textbf{Algebraic geometry: } A disproof of the Borho-Brylinski and Joseph characteristic cycles conjecture (1984).
\item \textbf{Combinatorics:} A proof of the positivity of parabolic Kazhdan-Lusztig polynomials  for any Coxeter system and any parabolic group.
\item \textbf{Knot theory: } A categorification of Jones polynomials and HOMFLYPT polynomials.
\item \textbf{Higher representation theory:} A disproof of the analogue for KLR algebras of James conjecture, by Kleschev and Ram (2011).
\item \textbf{Lie algebras: } An algebraic proof of Jantzen's conjecture about the Jantzen filtration in Lie algebras (1979).
\item \textbf{Combinatorics} A proof of the Monontonicity conjecture  (1985 aprox.)
\item \textbf{Combinatorics: }A proof of the Unimodality of structure constants in Kazhdan-Lusztig theory.

\end{enumerate}

We will start this story by the first of three levels, the classical one.

\subsection{Acknowledgements} This work is supported by the Fondecyt project 1160152 and the Anillo project ACT 1415 PIA Conicyt. The author would like to thank very warmly Macarena Reyes for her hard  work in doing most of the pictures. I  would also like to thank David Plaza, Paolo Sentinelli, Sebasti\'an Cea and Antonio Behn for detailed comments and suggestions.

\section{Classical level: Coxeter systems}

\subsection{Some definitions}\label{Some definitions}

 A \emph{Coxeter matrix} is a symmetric matrix   with entries in $\{1,2,\ldots\}\cup \{\infty\}$,  diagonal entries $1$ and off-diagonal entries at least $2$.

\begin{defi} A pair $(W,S)$, where $W$ is a group and $S$ is a finite subset of $W,$  is called a \emph{Coxeter system} if $W$ admits a presentation by generators and relations given by
$$ \langle s\in S\ \vert \ (sr)^{m_{sr}}=e\ \ \mathrm{if}\  s,r \in S \ \mathrm{and}\  m_{sr} \ \mathrm{is}\ \ \mathrm{finite} \rangle,$$
where $(m_{sr})_{s,r\in S}$ is a Coxeter matrix and $e$ is the identity element.
\end{defi}

We  then say  that $W$ is a \emph{Coxeter group}.
One can prove that in the Coxeter system defined above, the order of the element $sr$ is $m_{sr}$ (it is obvious that it divides $m_{sr}$). 
The \emph{rank} of the Coxeter system is the cardinality of $S$.
If $s\neq r,$ the relation $(sr)^{m_{sr}}=e$ is equivalent to 
$$\underbrace{srs\cdots}_{m_{sr}}=\underbrace{rsr\cdots}_{m_{sr}}$$
This is called a \emph{braid relation}. On the other hand, as $m_{ss}=1,$ we have that $s^2=e.$ This is called a \emph{quadratic relation}.
An \emph{expression}  of an element $x\in W$, is a tuple $\underline{x}=(s,r,\ldots, t)$ with $s,r,\ldots, t \in S$ such that $x=sr\cdots t$. The expression is \emph{reduced} if the length of the tuple is minimal. We denote $l(x)$ this length. 


\subsection{Baby examples}  We start with two baby examples of Coxeter systems. In these two cases (as well as in examples (C) and (D)) we will  calculate explicitly the two key objects in the theory, namely  the Kazhdan-Lusztig basis and the indecomposable Soergel bimodules.   

\newpage

\begin{description}

\item[(A)] The group $\mathrm{Symm}(\Delta)$ of symmetries of an equilateral triangle  is isomorphic to the group with $6$ elements
$$ \langle s, r\ \vert \ s^2=r^2=e, \ srs=rsr\rangle.$$
One possible isomorphism is given by the map

\begin{figure}[H] 
\begin{center}
 \includegraphics[scale=0.32]{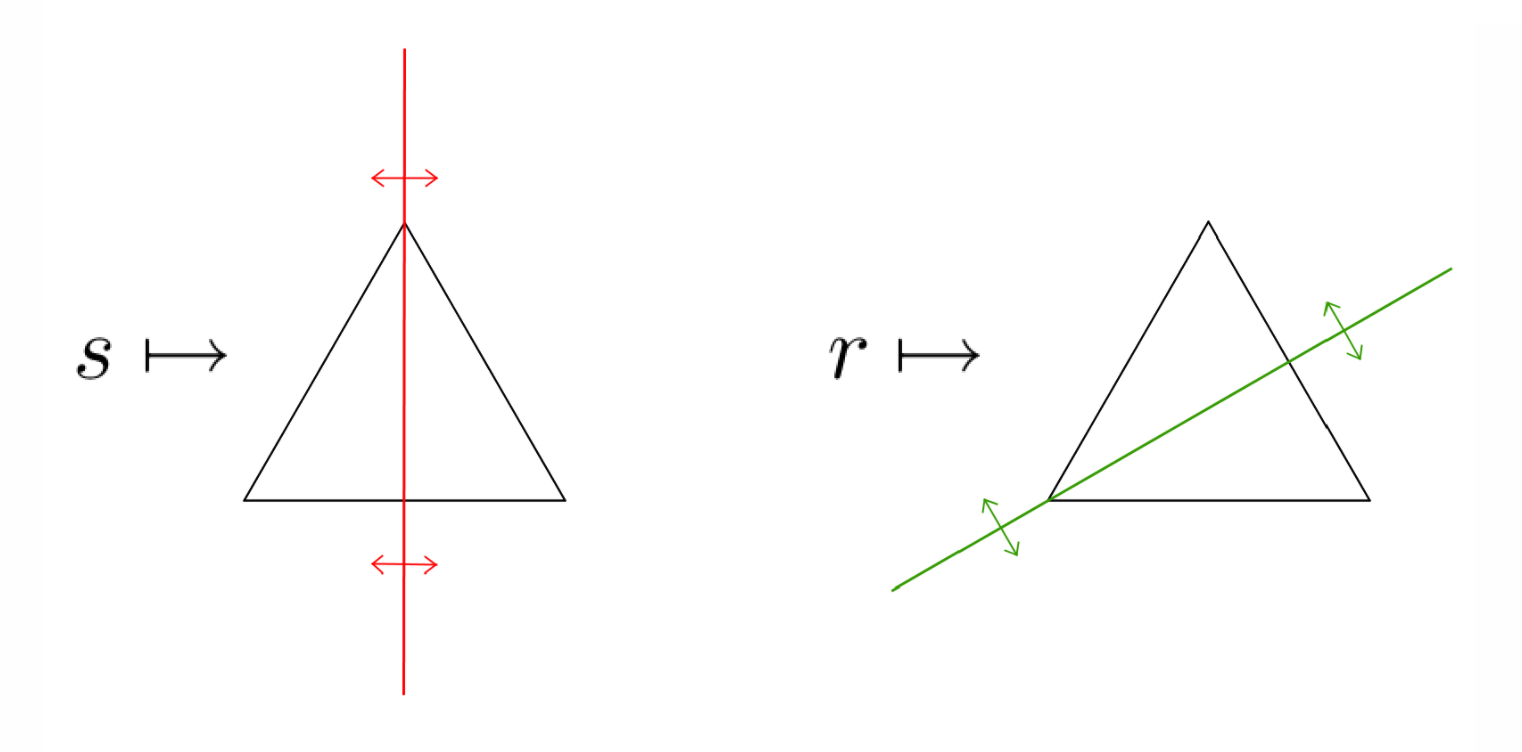} 
\end{center}
\end{figure}

\item[(B)] The group $\mathrm{Symm}(\square)$ of {symmetries of a square}  is isomorphic to the group with $8$ elements
$$ \langle s, r\ \vert \ s^2=r^2=e, \ srsr=rsrs\rangle.$$
One isomorphism between these groups is given by 
\begin{figure}[H] 
\begin{center}
 \includegraphics[scale=0.3]{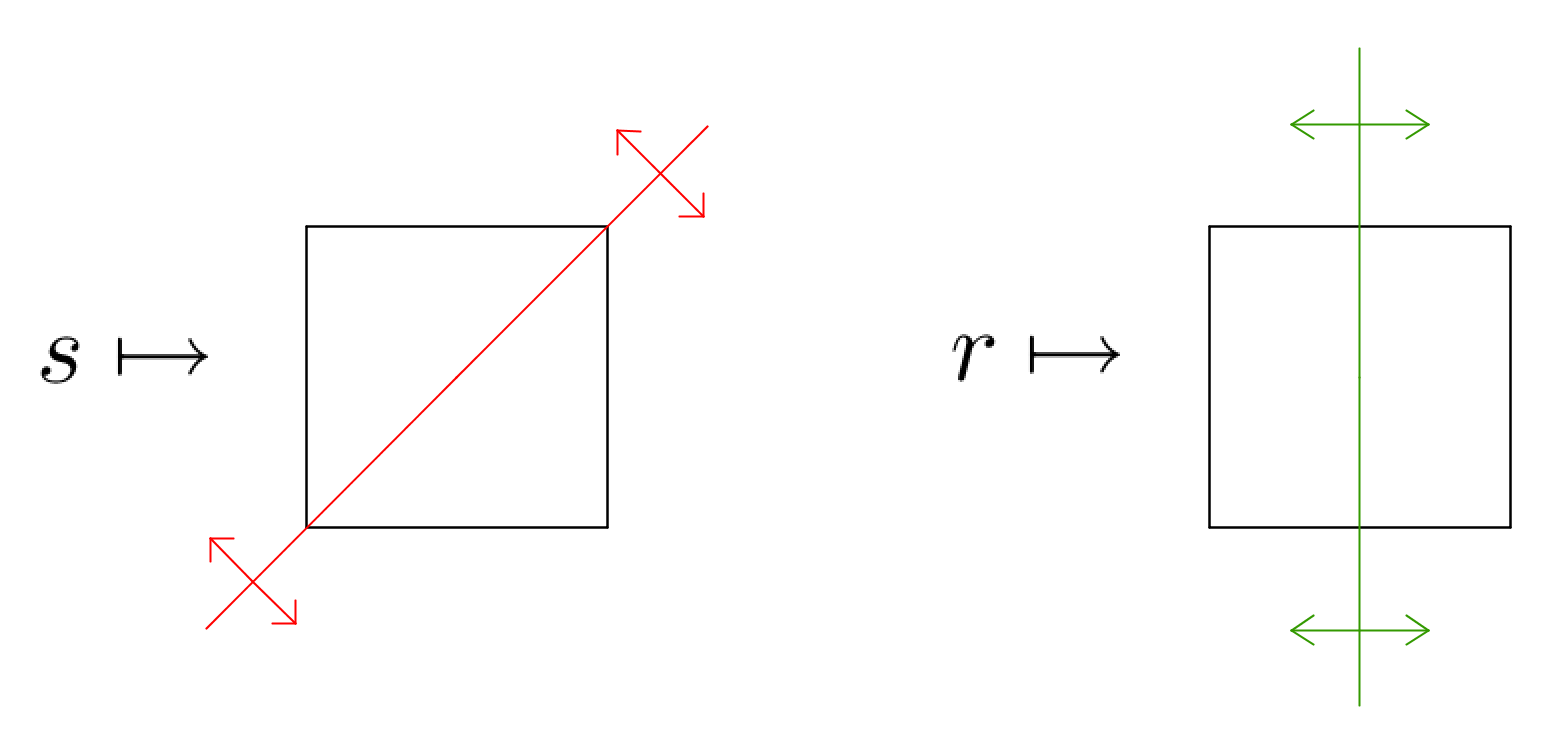} 
\end{center}
\end{figure} 

\end{description}

\subsection{Generalizing the baby examples: (infinite) regular polygons}

 One natural way to generalize the  baby examples is to consider the symmetries of a regular $n$-sided polygon. This is also a finite Coxeter group denoted $I_2(n)$ (the subindex $2$ in this notation denotes the rank of the Coxeter system, as defined in Section \ref{Some definitions}). A presentation of this group is given by
\begin{description}
\item[(C)]  \hspace{3cm} $ I_2(n)= \langle s, r\ \vert \ s^2=r^2=e, \ (sr)^{n}=e\rangle,$
\end{description} 
so $\mathrm{Symm}(\Delta)\cong I_2(3)$ and $\mathrm{Symm}(\square)\cong I_2(4).$ Again, one isomorphism here is given by sending $s$ to some reflection and $r$ to any of the ``closest reflections to it'' 
\begin{figure}[H] 
\begin{center}
 \includegraphics[scale=0.28]{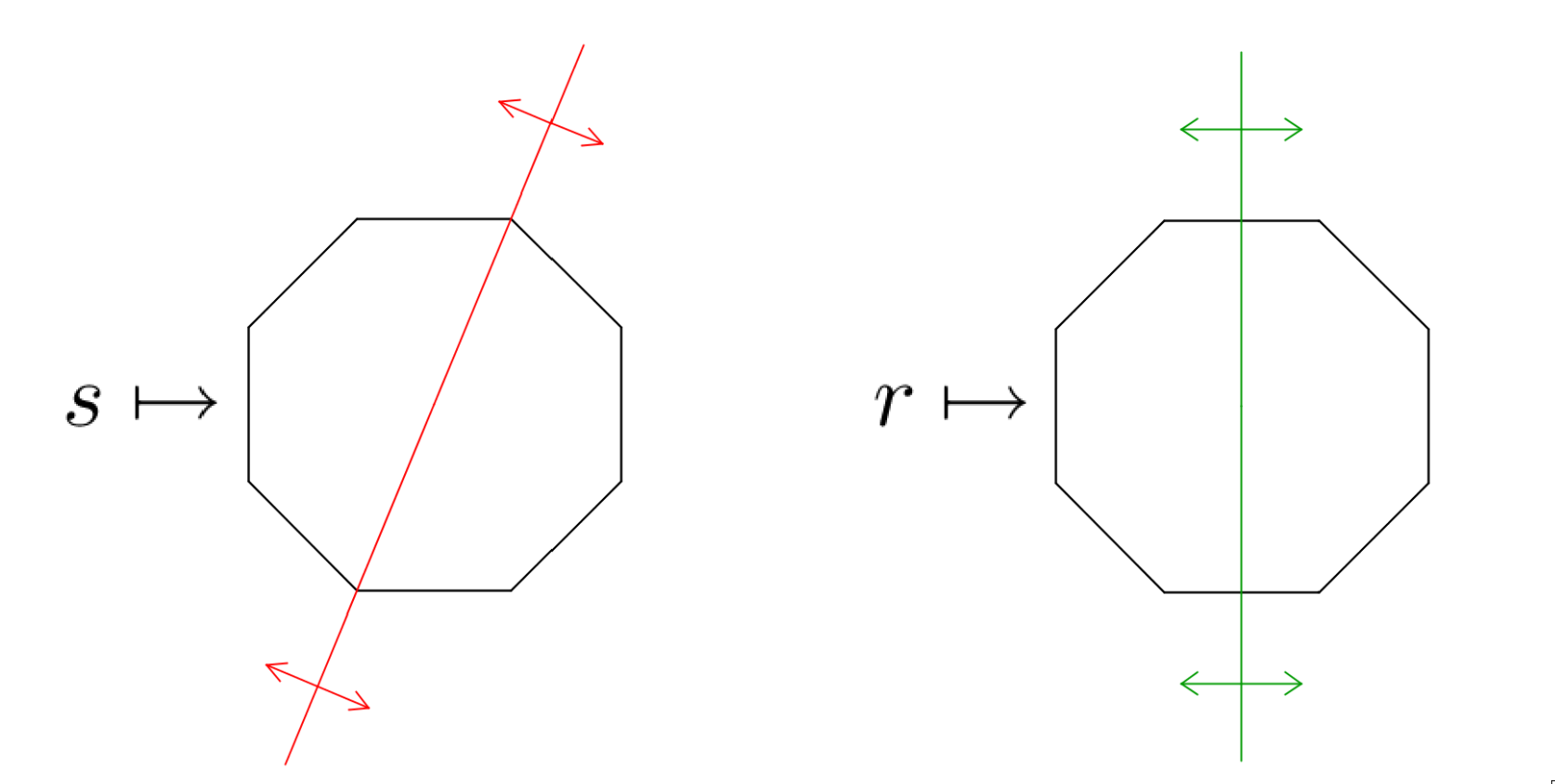} 
\caption{Example with $n=8$}  
\end{center}
\end{figure}

What is the \emph{Infinite regular polygon},  the ``limit" in $n$ of the groups $I_2(n)$? A reasonable way to search for a geometric  limit of the sequence of $n-$sided regular polygons, is to picture this sequence having ``the same size'' as in the figure
 \begin{figure}[H] 
 \includegraphics[width=\linewidth]{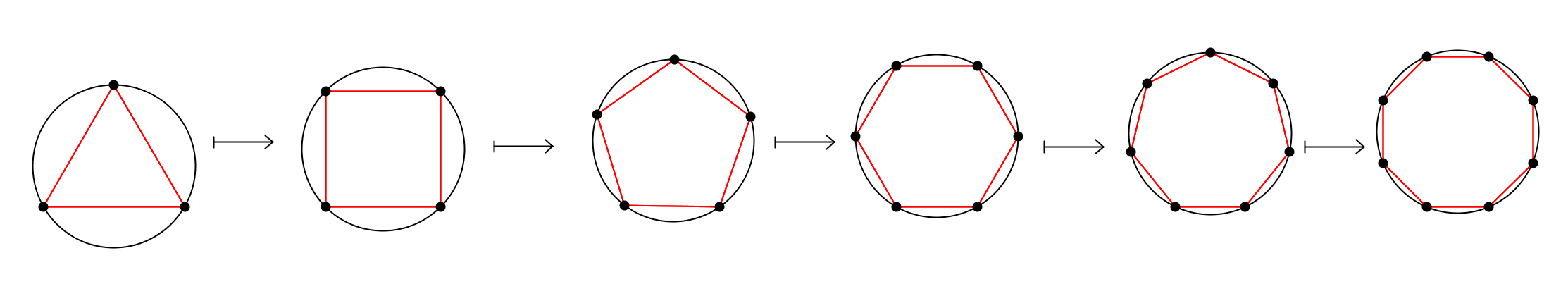} 
\caption{The incorrect mental image}  
\end{figure} 
If we do so, the (pointwise) limit is a circle. On the other hand, the limit when $n$ goes to infinity of $I_2(n)$ is algebraically clear if we consider  the presentation given in Example (C). It is the infinite group 

\begin{description}
\item \hspace{4cm} $ U_2=\langle s, r\ \vert \ s^2=r^2=e\rangle.$

\end{description} 

 But the geometric and algebraic descriptions given here do not coincide! The group $\mathrm{Symm}(\bigcirc)$ of symmetries of the circle (usually called the orthogonal group $\mathrm{O}(2)$) is not even finitely generated. We have passed to a continuous group! The best we can do is to see $  U_2$ as a dense subgroup of  $\mathrm{Symm}(\bigcirc)$ (just consider $s$ and $r$ to be two "random" reflections of the circle). 

There is a beautiful way to solve this problem: there is a discrete   geometric  limit of the sequence of $n-$sided regular polygons. Let us picture this sequence with one fixed side (in the figure, the darker one). Let us suppose that the fixed side has vertices in the points $(0,0)$ and $(0,1)$ of the plane.  
 \begin{figure}[H] 
 \includegraphics[width=\linewidth]{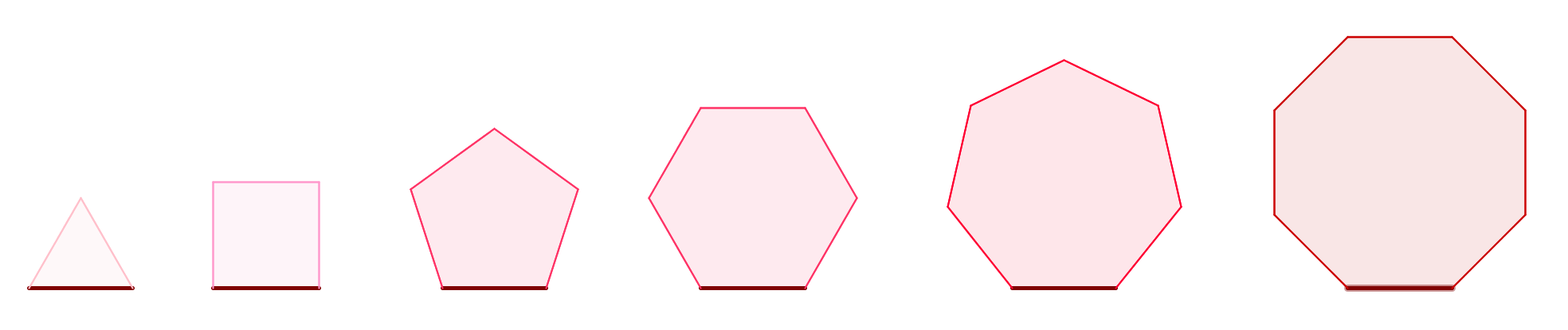} 
\caption{A new sequence}  

\end{figure} 
Then, if we make a ``close up'' around the darker side, and we put all the polygons together, we see
 \begin{figure}[H] 
 \includegraphics[width=\linewidth]{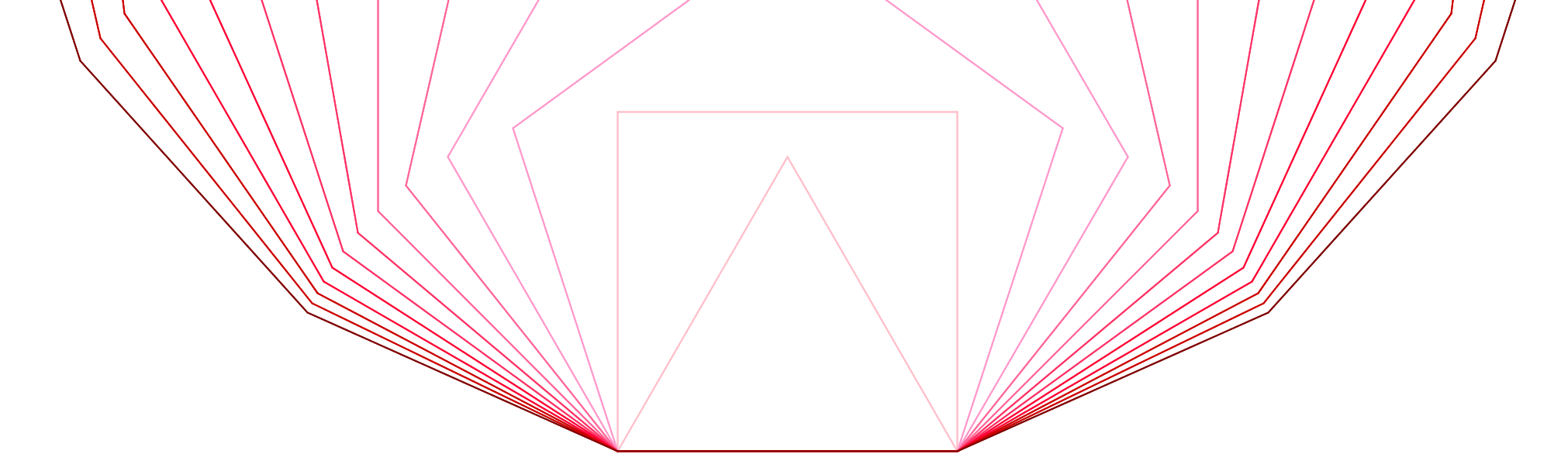} 
\caption{The figure opens like flower}  

\end{figure} 
The (pointwise) limit is $\mathbb{R}$ with the vertices converging to the set $\mathbb{Z}$. The symmetries of this geometric object (that we call $(\mathbb{R}, \mathbb{Z})$) is isomorphic to $U_2$! One isomorphism is given by sending $s$ to the reflection through $0$ and $r$ to the reflection through $1/2$ (again ``the closest reflection''). Now our geometric limit and the algebraic limit coincide and we can regain our lost calm. 

For our fourth example, we just   just raise the rank of $U_2$. 

\begin{description}
\item[(D)]  The \emph{Universal Coxeter system of rank $n$} is the group 
$$ {U}_n=\langle s_1, s_2, \ldots, s_n\ \vert \ s_1^2=s_2^2=\cdots=s_n^2=e\rangle$$
\end{description} 

This is the most complicated family of groups in which one can still compute all of the Kazhdan-Lusztig theory explicitly.

\subsection{Generalizing the baby examples:  tesselations}

We will denote by $[n,p,q]$ the Coxeter system with simple reflections $S=\{s,r,t\}$ and with Coxeter matrix  given by $m_{sr}=n, m_{st}=p$ and $m_{rt}=q.$

\subsubsection{Tesselations of the Euclidean plane}
A natural way to generalize the equilateral triangle is to consider the following tesselation of the Euclidean plane by equilateral triangles

\begin{figure}[H] 
\begin{center}
 \includegraphics[scale=0.19]{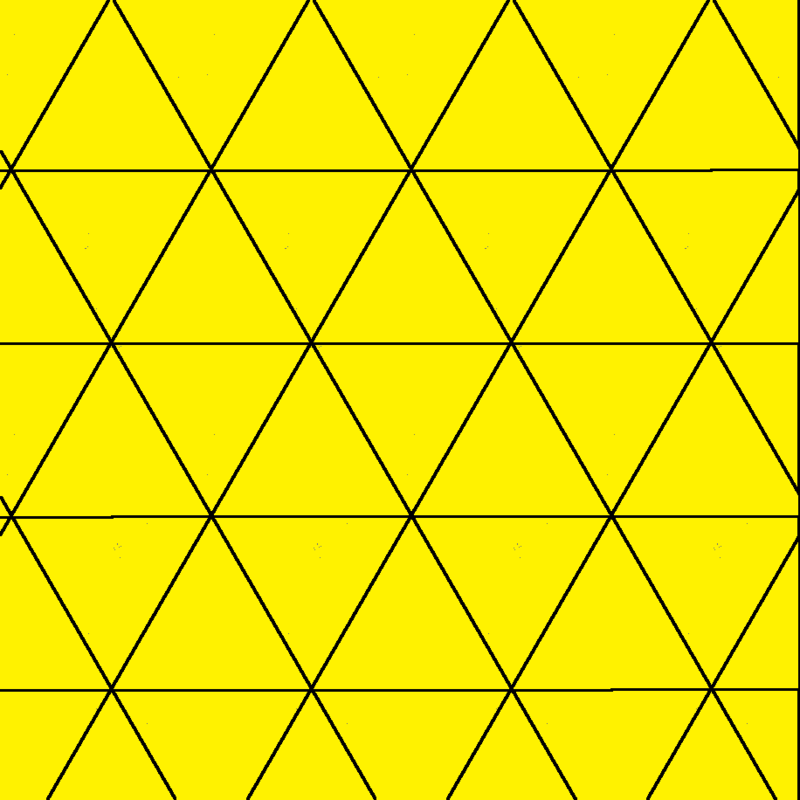}
\caption{[6,3,2]}  
\end{center}
\end{figure}

This tesselation can be generalized by coloring this tiling as in Figure \ref{triang} or by tiling the plane with other (maybe colored) regular convex polygons as it is the checkboard of Figure \ref{check} or the  honeycomb of Figure \ref{honey}.\begin{figure}[H]
\begin{minipage}[t]{0.30\linewidth}
    \includegraphics[width=\linewidth]{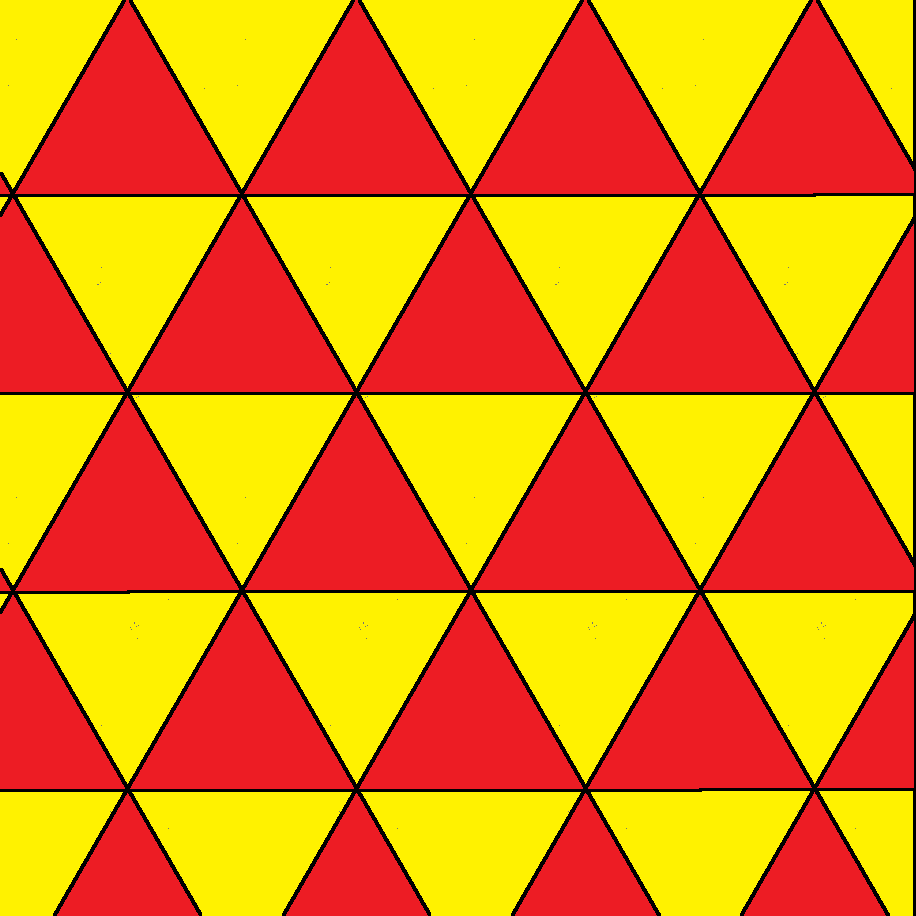}
    \caption{[3,3,3]}
    \label{triang}
\end{minipage}
    \hfill
\begin{minipage}[t]{0.30\linewidth}
    \includegraphics[width=\linewidth]{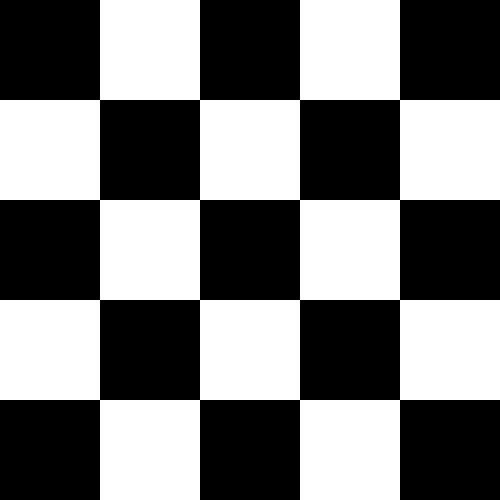}
    \caption{ [4,4,2]}
    \label{check}
\end{minipage} 
   \hfill
\begin{minipage}[t]{0.30\linewidth}
    \includegraphics[width=\linewidth]{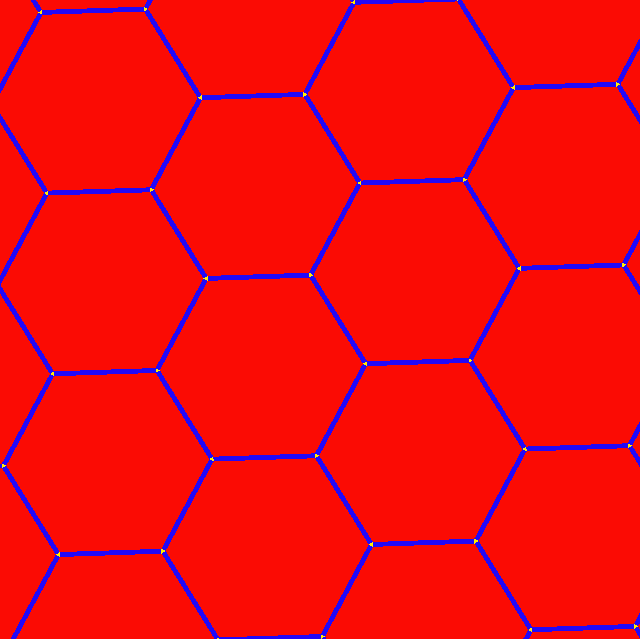}
    \caption{[6,3,2]}
    \label{honey}
\end{minipage} 
\end{figure}

Note the (curious?) equalities obtained by adding the corresponding inverses  of $n, p$ and $q$ in the last figures: $$\frac{1}{3}+ \frac{1}{3}+\frac{1}{3}=\frac{1}{4}+ \frac{1}{4}+\frac{1}{2}=\frac{1}{6}+ \frac{1}{3}+\frac{1}{2}=1$$

\subsubsection{Tesselations of the hyperbolic plane}\label{thp}

It is quite fantastic the amount of tesselations  by triangles of the hyperbolic plane. Here we give some examples using the Poincar\'e disk model. 
\begin{figure}[htbp]
\begin{minipage}[t]{0.30\linewidth}
    \includegraphics[width=\linewidth]{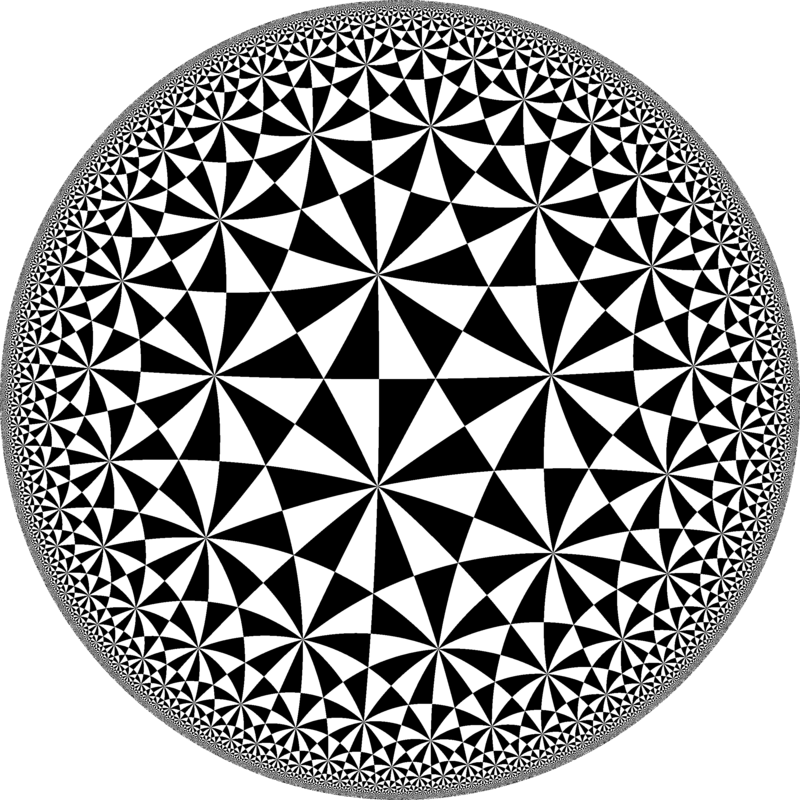}
    \caption{[7,3,2]}
    \label{trian}
\end{minipage}
    \hfill
\begin{minipage}[t]{0.30\linewidth}
    \includegraphics[width=\linewidth]{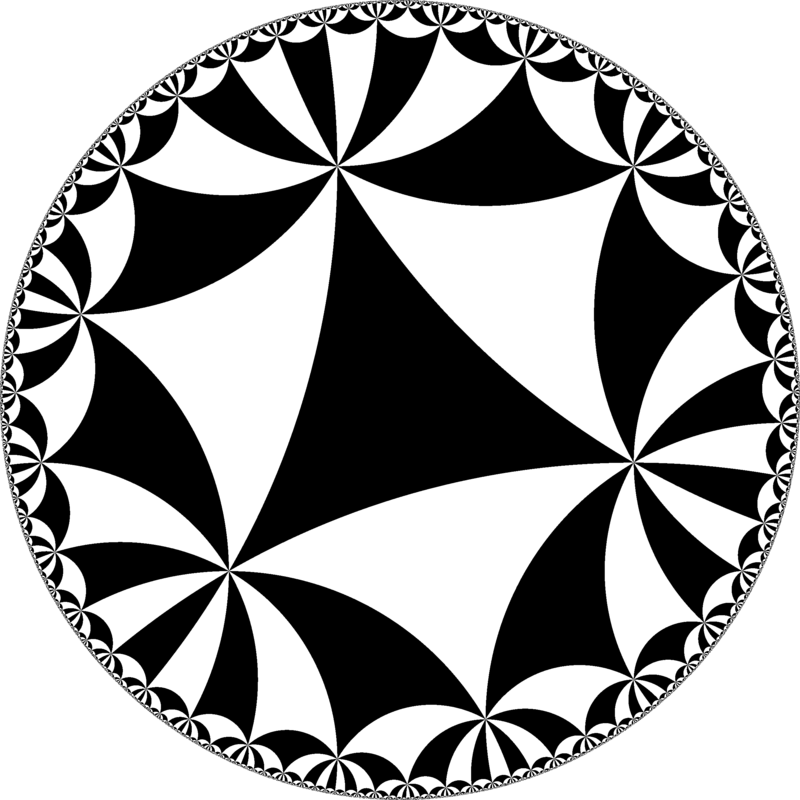}
    \caption{ [6,6,6]}
    \label{seis}
\end{minipage} 
   \hfill
\begin{minipage}[t]{0.30\linewidth}
    \includegraphics[width=\linewidth]{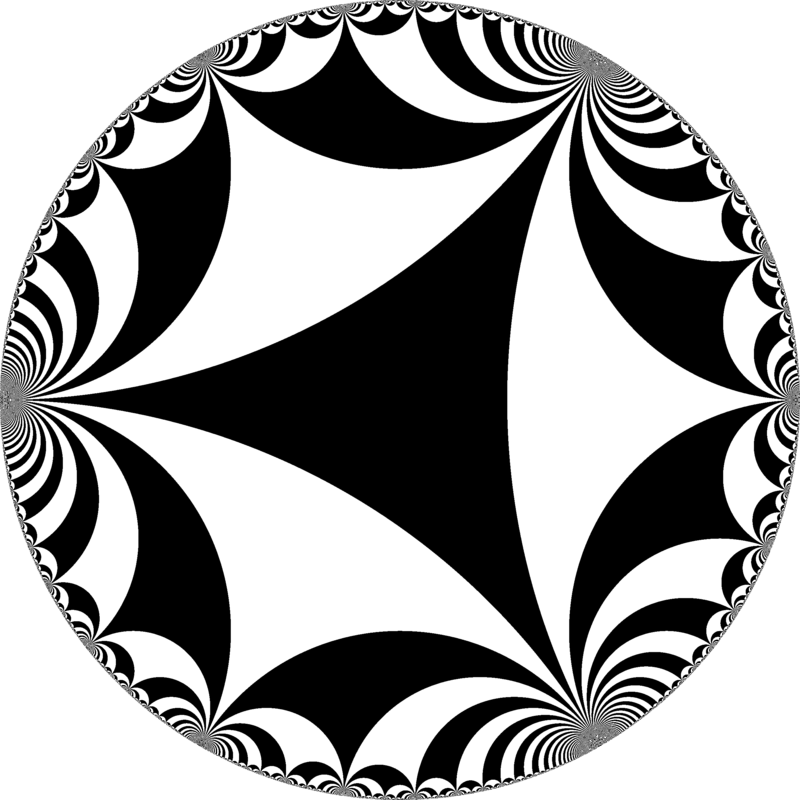}
    \caption{$[\infty,\infty,\infty,]$}
    \label{in}
\end{minipage} 
\end{figure}

Figure \ref{seis} is called by some (at least by me) ``Devil's tesselation''.

In general, there is a tesselation  by triangles of the hyperbolic plane  with group of symmetries $[n,p,q]$  if and only if \begin{equation}\label{brig}\frac{1}{n}+ \frac{1}{p}+\frac{1}{q}<1.\end{equation}
In particular, if  $n>3$ and $p,q\geq 3$ then this inequality is satisfied. So most rank three Coxeter groups are  the group of symmetries of some hyperbolic tiling. Those Coxeter groups which are not, are either the  group of symmetries of a tesselation by triangles of the Euclidean plane\footnote{Essentially those that can be constructed from Figures \ref{triang}, \ref{check} and \ref{honey}. The groups of symmetries appearing in this fashion are just the three groups $[3,3,3], [4,4,2]$ and $[6,3,2]$. } where inequality (\ref{brig}) is changed by an equality or they are the  group of symmetries of a tesselation by triangles of the sphere like the following

\begin{figure}[H]
\begin{minipage}[t]{0.30\linewidth}
    \includegraphics[width=\linewidth]{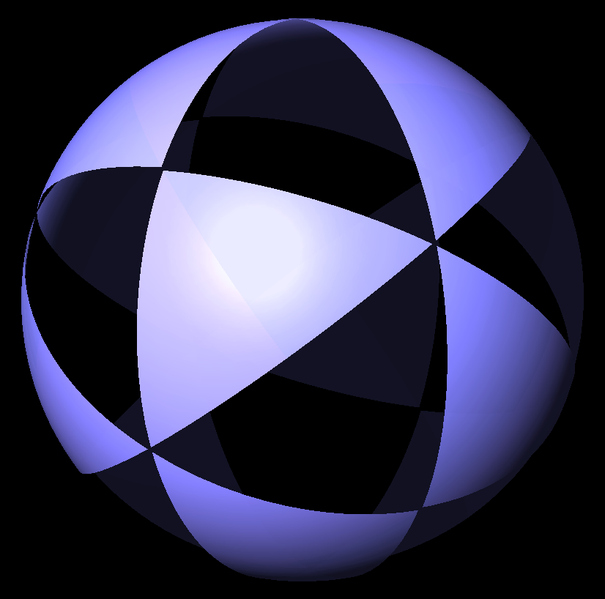}
    \caption{[2,3,3]}
    \label{12}
\end{minipage}
    \hfill
\begin{minipage}[t]{0.30\linewidth}
    \includegraphics[width=\linewidth]{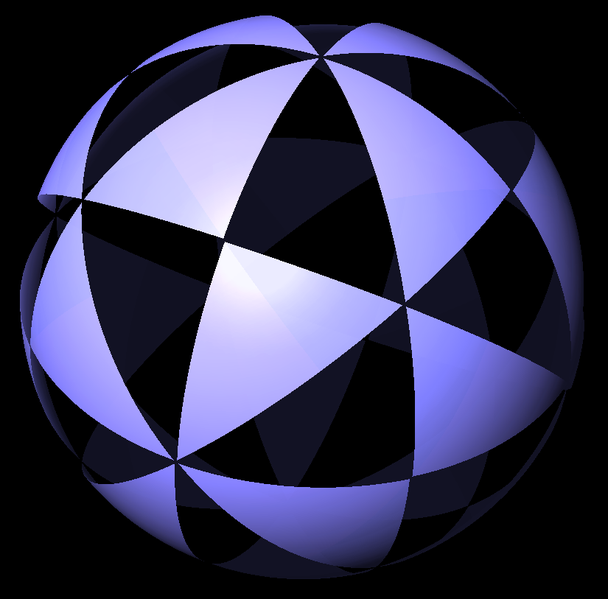}
    \caption{ [2,3,4]}
    \label{13}
\end{minipage} 
   \hfill
\begin{minipage}[t]{0.30\linewidth}
    \includegraphics[width=\linewidth]{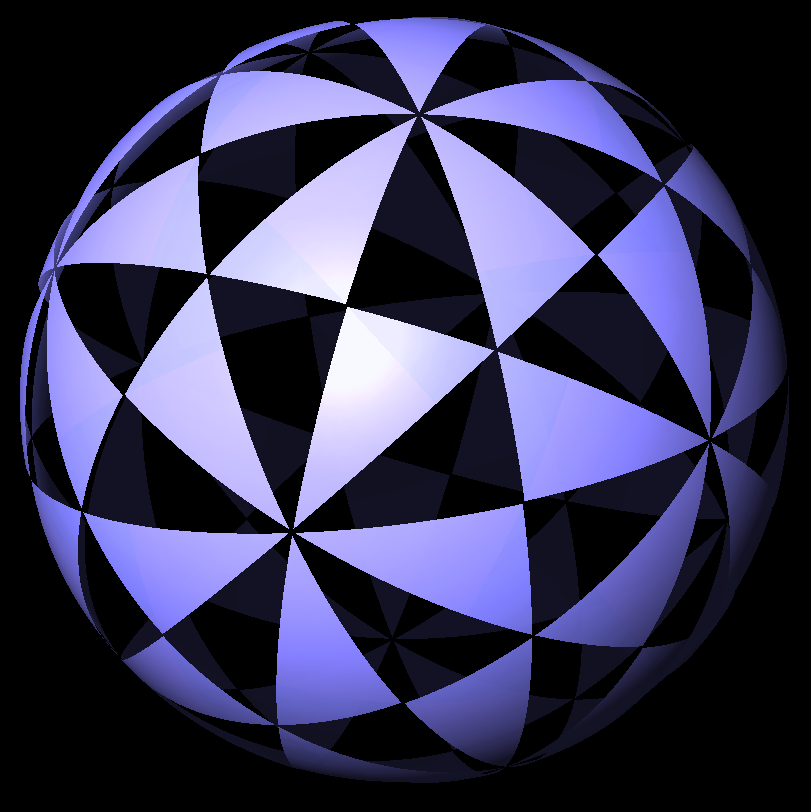}
    \caption{$[2,3,5]$}
    \label{14}
\end{minipage} 
\end{figure}

In the spherical case  the reversed inequality is satisfied

 \begin{equation}\frac{1}{n}+ \frac{1}{p}+\frac{1}{q}>1.\end{equation}

so the triples $[n,p,q]$ that appear here are the``little numbers'' (we just need to add to Figures \ref{12}, \ref{13} and \ref{14}, the groups $[2,2,n]$ for $n\geq 2$).

 \subsection{Generalizing the baby examples: more dimensions}

If we raise the rank, we can generalize our baby examples in a different way. We give the $n$-analogue of the baby examples.

\begin{enumerate}
\item The \emph{Weyl group of type $A_{n-1}$}. It can be defined as the symmetries  of an $n$-simplex, or equivalently, as $S_n$, the symmetric group in $n$ elements. The isomorphism between these two groups is obvious. It  admits a Coxeter presentation given by generators $s_i, 1\leq i <n$ and relations
\begin{itemize}
\item $s_i^2=e$ for all $i$.  
\item  $s_is_j =s_js_i$ if $ \vert i-j\vert \geq 2$ 
\item $s_is_{j}s_i=s_{j}s_is_{j}$ if $ \vert i-j\vert =1$ (i.e. $i=j\pm 1$)
\end{itemize}
The isomorphism from this group to the symmetric group is given by sending $s_i$ to the transposition $(i, i+1)\in S_n.$ 

\item The \emph{Weyl group of type $BC_{n-1}$}. It is the group of symmetries of an $n$-hypercube. It has order $2^nn!$

\item More generally, all symmetry groups of regular polytopes are finite Coxeter groups. Dual polytopes have the same symmetry group. 
\end{enumerate}

\subsection{More examples}
\begin{itemize}
\item Type A  and type B groups are examples of \emph{Weyl groups}. These groups appear in the theory of Lie algebras as the groups of symmetries of  root systems associated to semisimple Lie algebras over the complex numbers (so they are examples of finite reflection groups). There are three infinite families of Weyl groups, types $A_n, BC_n$ and $D_n$ ($n\in \mathbb{N}$) and the exeptional groups of type  $E_6$, $E_7,$  $E_8,$ $F_4,$ $G_2$. They are also symmetry groups of regular or semiregular polytopes. For example, this figure is a projection in the plane of an $8$-dimensional semiregular polytope with symmetry group $E_8$ 
\begin{figure}[H] 
\begin{center}
 \includegraphics[scale=0.43]{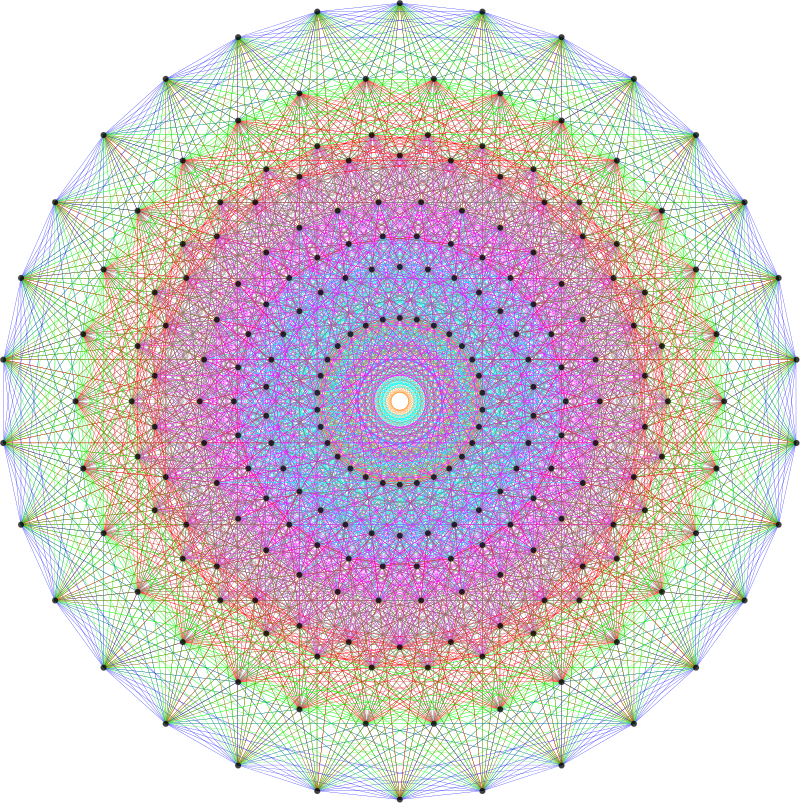} 
\end{center}
\end{figure}

\item The complete list of finite Coxeter groups is also known. Apart from the Weyl groups  there are the groups $H_2$ (symmetries of pentagon), $H_3$ (symmetries of the $3$-pentagon, or dodecahedron), $H_4$ (symmetries of the $4$-pentagon, or hecatonicosachoron, a polytope, not a dinosaur as one may think), and also the  infinite family $I_2(n)$ described above, although some of them are repeated: $I_2(3)=A_2,\,  I_2(4)=B_2,\,  I_2(5)=H_2$ and  $I_2(6)=G_2$.

\item The group $U_2$  is an example of an \emph{Affine Weyl group}, it is also called  $\widetilde{A}_1$. These groups appear naturally in the study of  representations of algebraic groups and they are semidirect products of a lattice and a Weyl group. Their classification is almost the same as the one for Weyl groups. The only difference is that the family of Weyl groups $BC_n$ gives rise to two families of affine Weyl groups, $\widetilde{B}_n$ and $\widetilde{C}_n,$ while each of the other Weyl groups, say $A_n, D_n, E_6,\ldots$ give rise to one affine Weyl group, namely   $\widetilde{A}_n, \widetilde{D}_n, \widetilde{E}_6, \ldots$ These groups also appear as symmetry groups of uniform tesselations. 

\item The \emph{Right-angled Coxeter groups} are the ones for which $m_{sr}$ is either $2$ or $\infty$ for each $s,r\in S.$ They are important groups  in geometric group theory. 
\item The \emph{Extra-large Coxeter groups} are the ones for which $m_{sr}\geq 4$  for each $s,r\in S.$  We saw lots of examples in section \ref{thp}.
\end{itemize}

\textbf{Personal philosophy of the author} \ We consider the right-angled and the extra-large Coxeter groups as extreme (and oposite) cases and Weyl groups as  being in the middle. Usually problems regarding Weyl groups are difficult to grasp while the same problems regarding the two mentioned cases are easier combinatorially and give light of the Weyl group case.

To learn more about Coxeter groups we advice the books \cite{Hu}, \cite{BB} and \cite{Da}.

\subsection{Bruhat order and an important property}\label{imp}

One important concept  about Coxeter systems is the \emph{Bruhat order}. It is defined by $x\leq y$ if some substring of some (or every) reduced word for $y$ is a reduced word for $x$. 

For example, in the case of $S_3$, the Bruhat order is represented in the following diagram 
\begin{figure}[H] 
\begin{center}
 \includegraphics[scale=0.25]{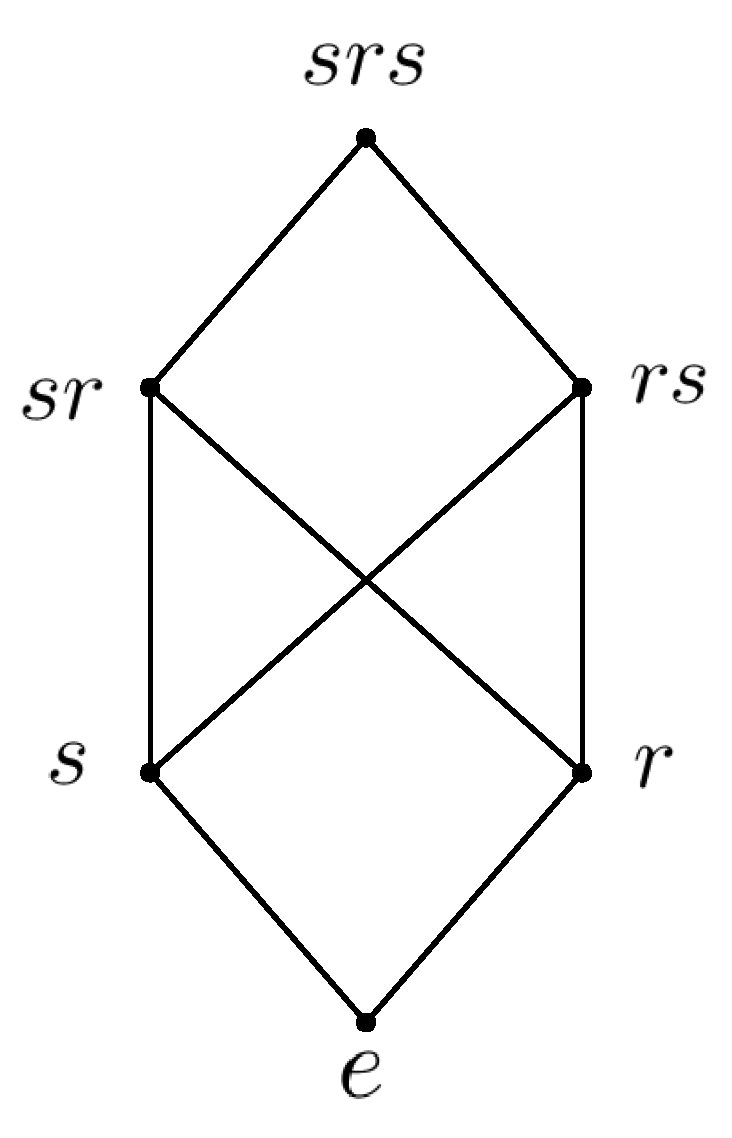} 
\end{center}
\end{figure}

We see that in this example, the only couples of non-comparable elements are $(s, r)$ and $(sr,rs)$.

A beautiful and important property (proved by  Hideya Matsumoto in 1964 \cite{Ma}) about Coxeter systems is that, if $x\in W$, one can obtain any reduced expression of $x$ from any other,  just by applying braid relations.  Moreover, if $sx<x$ there is an expression of $x$ that has $s$ in the left, i.e. $x$ admits an expression of the form $sr\cdots t$. Of course, if $xs<x$ then there is an expression of $x$ that has $s$ in the right. 

\section{Quantum level: Hecke algebras}
\subsection{Kazhdan-Lusztig's theory}

For the basic definitions of Hecke algebras and Kazhdan-Lusztig polynomials we follow \cite{SoKL}. Let $(W,S)$ be a Coxeter system. 

\begin{defi}
The \emph{Hecke algebra} $\mathcal{H}$ of a Coxeter system $(W,S)$ is the $\mathbb{Z}[v,v^{-1}]-$algebra with generators $h_s$ for $s\in S$ and relations 
\begin{itemize}
\item $h_s^2= (v^{-1}-v)h_s +1$ (quadratic relation)
\item  $\underbrace{h_sh_rh_s\ldots}_{m_{sr}}=\underbrace{h_rh_sh_r\ldots}_{m_{sr}}$ for all $s,r\in S$ (braid relation)
\end{itemize}
\end{defi}

 When $v$ is replaced by $1$ in the definition, one obtains the algebra  $\mathbb{Z}W$. Thus we can see the Hecke algebra as a deformation of the group algebra. 

 For any reduced expression $\underline{s}=sr\ldots t$ of an element $x\in W$ define the element $h_{\underline{s}}=h_sh_r\cdots h_t$. By the forementioned result of Matsumoto \cite{Ma} we know that $h_{\underline{s}}$ does not depend on the reduced expression ${\underline{s}}$, it just depends on $x$. 
 We call this element $h_x$. We define $h_{e}=1$.  The following is a basic lemma. 

\begin{lem}[Nagayoshi Iwahori]\label{Iwa}  The set $\{h_x\}_{x\in W}$ is a basis of $\mathcal{H}$ as a $\mathbb{Z}[v,v^{-1}]-$algebra, called the \emph{standard basis}.  
\end{lem}

The element $h_s$ has an inverse, namely $(h_s+v-v^{-1})$ as it is shown in the following calculation.

\begin{eqnarray}
h_s(h_s+v-v^{-1})&=&[(v^{-1}-v)h_s +1]+h_s(v-v^{-1})\nonumber\\
&=&1
\end{eqnarray}

This implies that $h_x$ has an inverse for every $x\in W.$ So we can define a $\ZZ$-module morphism  $d:\mathcal{H}\rightarrow \mathcal{H}$ by the formula $d(v)=v^{-1}$ and $d(h_x)=(h_{x^{-1}})^{-1}.$ It is an exercice to prove that this is a ring morphism, and we call it the \emph{duality} in the Hecke algebra. 

Let us make a short calculation 
\begin{eqnarray}\label{bs}
d(h_s+v)&=& (h_s)^{-1} +v^{-1}\nonumber\\
&=&(h_s+v-v^{-1})+v^{-1}\nonumber\\
&=& (h_s+v)
\end{eqnarray}

So we obtain our first example of self-dual element. 
The following theorem (see  \cite{KL}) is the foundational theorem of Kazhdan-Lusztig theory.

\begin{thm}[David Kazhdan and George Lusztig]\label{thKL}
For every element $x\in W$ there is a unique self-dual element $b_x\in \mathcal{H}$ such that 
\begin{equation}b_x\in h_x+\sum_{y\in W}v\Z[v]h_y.\end{equation}
The set $\{b_x\}_{x\in W}$  is a $\Z[v,v^{-1}]-$basis of $\mathcal{H},$ called the Kazhdan-Lusztig basis. If we write $b_x= h_x+\sum_{y\in W}h_{y,x}h_y$ then the Kazhdan-Lusztig polynomials $p_{y,x}$ are defined by the formula $p_{y,x}=v^{l(x)-l(y)}h_{y,x}.$
\end{thm}
\begin{remark}\label{rem1} We will prove in Section \ref{proofth} a stronger version of this theorem, namely that \begin{equation}\label{KL}b_x\in h_x+\sum_{y<x}v\Z[v]h_y,\end{equation} where $<$ refers to the Bruhat order. 
\end{remark}

Before we prove this theorem we will believe it for a while and calculate the Kazhdan-Lusztig bases in some  examples. 

\subsection{Calculations of KL bases in examples}

In this section we give an explicit calculation of the Kazhdan-Lusztig basis for the examples A, B and C and we will give the formula (without a  proof) for example D. 

\subsubsection{Baby example A} It is clear that $b_{\mathrm{id}}=1$. We have seen that $h_s+v$ is self-dual and it is of the form (\ref{KL}), so $b_s=(h_s+v).$ By symmetry between $s$ and $r$ we have that $b_r=(h_r+v).$

It is easy to see that $b_{sr}=b_{s}b_{r}$. It is self-dual because $d$ is a ring morphism and $b_{s}$ and $b_{r}$ both are. On the other hand it is of the form (\ref{KL})
$$b_{s}b_{r}=h_{sr}+vh_s+vh_r+v^2,$$
and again we obtain $b_{rs}$ by symmetry. 

If we were very optimistic we would believe that $b_{srs}=b_{s}b_{r}b_{s}$, which is self-dual. Let us calculate
\begin{eqnarray}
b_{s}b_{r}b_{s}&=&(h_{sr}+vh_s+vh_r+v^2)(h_s+v)\nonumber\\
&=&(h_{srs}+\underbrace{vh_s^2}+vh_{rs}+v^2h_s)+(vh_{sr}+v^2h_s+v^2h_r+v^3)
\end{eqnarray}
But $vh_s^2=(1-v^2)h_s +v,$ so $b_{s}b_{r}b_{s}$ is not of the form (\ref{KL}), we have a term that is $h_s$. To solve this issue we substract  $b_{s}.$ We still have a self-dual element and we eliminate the $h_s$ from the sum, so finally we obtain
\begin{eqnarray}
b_{srs}&=&b_{s}b_{r}b_{s}-b_{s}\nonumber\\
&=&h_{srs}+vh_{rs}+vh_{sr}+v^2h_s+v^2h_r+v^3
\end{eqnarray}

\subsubsection{Baby example B} In Example A we never used that $m_{sr}=3$ in our calculations. The point is that if $m_{sr}\neq 3$ then $b_{s}b_{r}b_{s}-b_{s}\neq b_{r}b_{s}b_{r}-b_{r}$. In our baby example \textbf{(B)} we have $m_{sr}=4. $ For the same reasons as before we have 
\begin{itemize}
\item $b_{sr}=b_{s}b_{r}$ 
\item $b_{srs}=b_{s}b_{r}b_{s}-b_{s}$ \item $b_{rsr}=b_{r}b_{s}b_{r}-b_{r}.$ 
\end{itemize}
So we just need to calculate $b_{srsr}.$
We start again with $$b_{s}b_{r}b_{s}b_{r}=(h_{sr}+vh_s+vh_r+v^2)(h_{sr}+vh_s+vh_r+v^2).$$
When we expand the right hand side, the only terms in the sum that do not give elements  of the form (\ref{KL}) are $(vh_s)(h_{sr})=(1-v^2)h_{sr}+vh_r$ and $(h_{sr})(vh_r)=(1-v^2)h_{sr}+vh_s,$ so if we substract $2b_{sr},$ we eliminate the two ``problematic" terms and we obtain 

\begin{eqnarray}
b_{srsr}&=&b_{s}b_{r}b_{s}b_{r}-2b_{s}b_{r}\nonumber\\
&=&h_{srsr}+vh_{srs}+vh_{rsr}+v^2h_{sr}+v^2h_{rs} +v^3h_s+v^3h_r+v^4\nonumber\\
&=& \sum_{y\leq srsr}v^{4-l(y)}h_y
\end{eqnarray}

\subsubsection{Example C}  

As one might have conjectured looking at the first two examples, for $x\in I_2(n)$ we have
\begin{equation}\label{excc}b_x= \sum_{y\leq x}v^{l(x)-l(y)}h_y\end{equation}
If we would want to play the game we played in the first two examples, this is, if we wanted to express $b_x$ as additions and substractions of \emph{Bott-Samelsons} (i.e. objects of the type $b_sb_r\cdots b_t,$ with $s,r,\ldots, t\in S$) then we obtain the combinatorics appearing in Temperley-Lieb algebras, but this approach is a bit complicated. It is  easier to  prove directly  equation (\ref{excc}).

 We will prove equation (\ref{excc}) by induction on the length of $x$. Let us call for the moment \begin{equation*}\label{exc}c_x:= \sum_{y\leq x}v^{l(x)-l(y)}h_y\end{equation*}
 
 So our induction hypothesis is that $b_y=c_y$ for all elements $y$ such that $l(y)\leq n$. Let us introduce the following 

\begin{nota} In the group $I_2(n)=\langle s,r \rangle $ we will denote, for any $0\leq i\leq n$,   $$s(i):=\underbrace{srs\cdots}_{i\ \mathrm{terms}}\in W \hspace{.3cm}  \mathrm{and} \hspace{.3cm} r(i):=\underbrace{rsr\cdots}_{i\  \mathrm{terms}} \in W  $$
\end{nota}
We will prove that 
\begin{equation}\label{good}
c_{s(n+1)}=b_sc_{r(n)}-c_{s(n-1)}
\end{equation}
thus mimicking the construction that we will do in the proof of Theorem \ref{thKL}. Once we have proved this, we are done, because the right-hand side is clearly self-dual (by induction hypothesis), and the left-hand side clearly belongs to $h_{s(n+1)}+\sum_yv\ZZ[v]h_y$ and thus $c_{s(n+1)}=b_{s(n+1)}$.

We call $x:=s(n+1).$  If $$A:=\sum_{0\leq i \leq n}v^{n-i}h_{r(i)}$$ and 
$$B:=\sum_{0< i < n}v^{n-i}h_{s(i)},$$
then $c_{r(n)}=A+B.$

We have that 
\begin{align*}
b_sA&=\sum_{0\leq i \leq n}v^{n-i}h_{s(i+1)}+\sum_{0\leq i \leq n}v^{(n+1)-i}h_{r(i)}\\
&=  \sum_{1\leq j \leq n+1}v^{(n+1)-j}h_{s(j)} +\sum_{0\leq i \leq n}v^{(n+1)-i}h_{r(i)}\\
&=c_{s(n+1)}.
\end{align*}
On the other hand we have, using the quadratic relation, 
\begin{align*}
b_sB&=\sum_{0< i < n}v^{n-i}h_{r(i-1)}+\sum_{0< i < n}v^{n-i-1}h_{s(i)}\\
&=   \sum_{0\leq j < n-1}v^{(n-1)-j}h_{r(j)}+\sum_{0< i < n}v^{(n-1)-i}h_{s(i)} \\
&=    c_{s(n-1)}
\end{align*}
  thus proving equation (\ref{good}).

\begin{remark}\label{univ}
We can see that the  proof of equation (\ref{good}) is independent of the dihedral group in which one is placed, but one knows that in $I_2(n)$ we have  $h_{s(n)}=h_{r(n)}$, and thus, by equation (\ref{good}) we have $b_{s(n)}=b_{r(n)}$. This is the basic insight of the following example. 
\end{remark}

\subsubsection{Example D}  
Kazhdan-Lusztig polynomials were discovered (or invented) in 1979, but it was only in 1990 that Matthew Dyer gave a formula  \cite{Dy} to calculate inductively the Kazhdan-Lusztig basis for a Universal Coxeter system (we will not reproduce here the proof, although it is not a difficult one). In this case, every element has only one expression as a product of elements of $S$.

\begin{thm}[Dyer's Formula]\label{Dy} Let $x\in U_n$ and $x=rs\cdots $ with $r,s, \ldots \in S.$ Then we have the following recursive formula
\begin{eqnarray}\label{universal}
b_rb_x  & = & (v+v^{-1})b_x \nonumber\\ 
b_tb_x&=&b_{tx}\hspace{1.7cm} \mathrm{if} \ t\neq r,s \nonumber \\
b_sb_x &=&b_{sx}+b_{rx}\ \hspace{.6cm}  \mathrm{if}\ \ s\neq r 
\end{eqnarray}
\end{thm}


Our baby examples A and B are particular cases of example C (the case of dihedral groups). We just saw in equation (\ref{good}) that Dyer's Formula is also true for Dihedral groups (the first case of the formula comes from the fact $b_sb_s=(v+v^{-1})b_s$ and the second case never appears in Dihedral groups). In fact the same proof works for the infinite Dihedral group as we said in Remark \ref{univ}. So the whole point of this theorem is that the calculation of the Kazhdan Lusztig basis in this case is local in nature, i.e. different strings of alternating simple reflections ``don't intersect each other'' in the following sense.

Let $x=\cdots psps$ be an alternating sequence of simple reflections ending by $s$, and $y=srsrs\cdots$ an alternating sequence of simple reflections starting by $s$. Then we use the following notation
$$b_x\ast b_y:= \frac{b_xb_y}{(v+v^{-1})}$$
It is an exercice to prove that $b_x\ast b_y\in \cH$ (hint: $b_x$ is divisible on the right and $b_y$ is  divisible on the left by $b_s$).     Consider the element $$x:={\color{red}s}{\color{blue}r}{\color{red}s}{\color{blue}r}{\color{red}s}{\color{blue}r}{\color{orange} p}{\color{blue}r}{\color{orange} p}{\color{blue}r}{\color{teal} u}{\color{olive}q}{\color{darkgray} t}{\color{olive}q} {\color{blue}r}{\color{olive}q}{\color{blue}r}{\color{olive}q}{\color{blue}r}{\color{olive}q}.$$ 
As we said, the calculation is local, we have
 $$ b_x=b_{{\color{red}s}{\color{blue}r}{\color{red}s}{\color{blue}r}{\color{red}s}{\color{blue}r}}\ast b_{{\color{blue}r}{\color{orange} p}{\color{blue}r}{\color{orange} p}{\color{blue}r}}\ast b_{{\color{blue}r}{\color{teal} u}}\ast b_{{\color{teal} u}{\color{olive}q}} \ast b_{{\color{olive}q}{\color{darkgray} t}{\color{olive}q}}\ast b_{{\color{olive}q}{\color{blue}r}{\color{olive}q}{\color{blue}r}{\color{olive}q}{\color{blue}r}{\color{olive}q}}.$$

 With Dyer's Formula in hand we could have done all the calculations we did in the other examples with no effort. We apply it three times and we obtain
\begin{itemize}
\item $b_{sr}=b_sb_r$ 
\item$b_{sr}b_s=b_{srs}+b_s \longrightarrow b_{srs}=b_{s}b_rb_s-b_s$ 
\item $b_{srs}b_r=b_{srsr}+b_{sr} \longrightarrow  b_{srsr}=b_sb_rb_sb_r-2b_sb_r$
\end{itemize}
\hspace{10cm}Aha!

\subsection{Proof of Theorem \ref{thKL}}\label{proofth}

\begin{proof} We reproduce the  beautiful (and simple) proof  of  Soergel  \cite[Theorem 2.1]{SoKL} of the stronger version explained in Remark  \ref{rem1}. 

\subsubsection{Existence}\label{1}

We prove it by induction on the Bruhat order. It is clear that $b_e=h_e=1$ and we have already seen in equation (\ref{bs}) that $b_s=h_s+v.$ The following equation is easy (see the important property in Section \ref{imp})
\begin{equation}\label{cases}
    b_sh_x=
    \begin{cases}
      h_{sx}+vh_x & \text{if } sx > x\\
      h_{sx}+v^{-1}h_x & \text{if } sx < x.
    \end{cases}
\end{equation}
Now suppose we have proved the existence for all elements lesser than $x$ in the Bruhat order and $x\neq e.$ Then we can find an $s\in S$ such that $sx<x$. By  induction hypothesis and using equation (\ref{cases}) one has \begin{equation} b_sb_{sx}=h_x+\sum_{y<x} p_{y}h_y, \end{equation} for some $p_{y}\in \ZZ[v]$ (the $v^{-1}$ in equation (\ref{cases}) is the only problem). But if we define \begin{equation}b_x=b_sb_{sx}-\sum_{y<x} p_{y}(0)b_y,\end{equation}
we still obtain a self-dual element (a $\ZZ$-linear combination of self-dual elements) and it clearly is of the prescribed form, so this proves the existence of $b_x.$

\subsubsection{Uniqueness}

Suppose that we have two self-dual elements $c=h_x+h$ and $c'=h_x+h'$, with $h,h' \in \cH. $ Then one has that $h-h'$ is self-dual as $h-h'=c-c'.$

We just need to prove 
\begin{claim}
If $h\in \sum_yv\ZZ[v] h_y$ and $h$ is self-dual, then $h=0$. 
\end{claim}
Let us call $b_x$ the element that we constructed  in Section \ref{1}. It is easy to see, by equation (\ref{KL}) that $$h_x\in b_x+\sum_{y<x}\ZZ[v,v^{-1}]b_y,$$
thus \begin{equation}\label{d}d(h_x)\in b_x+\sum_{y<x}\ZZ[v,v^{-1}]b_y \subseteq  h_x+\sum_{y<x}\ZZ[v,v^{-1}]h_y.\end{equation}

Let us write $h=\sum_y p_yh_y$, with $p_y\in v\ZZ[v]$ and let $z$ be a maximal element (in the Bruhat order) such that $p_z\neq 0.$ In other words, $h=p_zh_z +\sum_{y\ngeq z} p_yh_y.$ By equation (\ref{d}) we obtain 
$$d(h)\in d(p_z)h_z +\sum_{y\ngeq z}\ZZ[v,v^{-1}] h_y.$$ As $h$ is self-dual, this implies that $d(p_z)=p_z,$ contradicting the fact that \newline $p_z\in v\ZZ[v]$ and thus proving the claim.

\subsubsection{ $\{b_x\}_{x\in W}$ is a basis}

The set $\{b_x\}_{x\in W}$ is a $\Z[v,v^{-1}]-$basis given that the set $\{h_x\}_{x\in W}$ is a basis and by using the triangularity property in Remark \ref{rem1}.
\end{proof}

\subsection{Our favorite open questions about Hecke algebras}

It is one of the most important achievements in the theory of Soergel bimodules the proof \cite{EW} by B. Elias and G. Williamson of the conjecture (stated by Kazhdan and Lusztig) that Kazhdan-Lusztig polynomials have positive coefficients. We will come back to this beautiful and fundamental result in a subsequent paper of this saga. For the moment, this theorem gives rise to the following question. 

\begin{ques}Give a combinatorial formula for the coefficients of Kazhdan-Lusztig polynomials, i.e. express the coefficients of $p_{x,y}$ as the cardinality of some combinatorially defined set (even for the symmetric group this would be extremely interesting).
\end{ques}

For the second question we have to introduce the \emph{Braid group} $B_W$ of a Coxeter system $(W,S)$. It is the group $$ \langle \sigma_s, s\in S\ \vert \ (\sigma_s\sigma_r)^{m_{sr}}=e\ \ \mathrm{if}\  s\neq r \in S \ \mathrm{and}\  m_{sr} \ \mathrm{is}\ \ \mathrm{finite} \rangle.$$
The fact that we impose $s\neq r$ is equivalent to not ask that $s^2=e$ for all $s\in S.$ Thus the braid group of $W$ is infinite unless $W$ is trivial.

\begin{ques}\label{braid} Is the following group morphism
\begin{align*}
B_W&\rightarrow \cH(W)\\
\sigma_s&\mapsto h_s
\end{align*}
an injection?
\end{ques}

There is a categorical version of this question due to Rouquier. He conjectures that ``the braid group injects in the 2-braid group''. We will explain this conjecture in detail in a subsequent paper. But we must say that  Question \ref{braid} implies the conjecture of Rouquier.

\section{Categorical level in the baby example A}\label{babyA}

\subsection{The objects: Soergel bimodules}

In this section we will introduce Soergel bimodules  in the baby example A, i.e. the symmetric group $S_3.$ This case is hard enough to start with and most of the features of general Soergel bimodules are already visible in this example. 

 Consider the polynomial ring $R=\mathbb{R}[x,y,z]$ (we could replace $\RR$ by any field of characteristic different from $2$ in this section and the results would stay true). We have a natural action of $S_3$ on $R$. The simple reflection $s$ interchanges $x$ and $y$. In formulas $$s\cdot f(x,y,z)=f(y,x,z).$$ The simple reflection $r$ interchanges $y$ and $z$.  So $R^s$ (the subset of $R$ fixed by the action of $s$) is the polynomial ring $\mathbb{R}[x+y, xy, z]$ and $R^r=\mathbb{R}[x, y+z,yz].$ The subring fixed by both simple reflections $s$ and $r$ (or, what is the same, by the whole group $S_3$) is $$R^{s,r}=\mathbb{R}[x+y+z,xy+xz+yz,xyz].$$ 

 If we want to enter into Soergel-bimoduland (this is an invented word) we have to take the grading into account. For technical reasons we need $R$ to have the usual $\mathbb{Z}$-grading multiplied by two. So $x,y$ and $z$ will be in degree $2$ (there are no elements of odd degree). The polynomials $x^2$ and $xz$ have degree $4$, the polynomial $3xy^2z^7$ has degree $20.$ And we define the ring $R$ shifted ``down'' by one $R(1)$ by declaring that $x$ is in degree $1$,  $x^2$ in degree $3$ and $3xy^2z^7$ in degree $19.$ Formally, if $B=\oplus_{i\in \mathbb{Z}}B_i$ is a graded object, we declare that the shifted object $B(m)$ in degree $i$ is $B(m)_i:=B_{m+i}.$

The $\mathbb{Z}$-graded $R-$bimodule $R$ is the easiest example of a Soergel bimodule.  The second example of a Soergel bimodule is  the  $\mathbb{Z}$-graded $R-$bimodules $B_s:=R\otimes_{R^s}R(1)$, for $s\in S$. Just for pedagogical reasons we insist that in $B_s$ the elements $x\otimes y$ and $z^2\otimes1 +1\otimes xz$ have degree $3$ and $x^3\otimes zx$ has degree $9.$

Another example is the product $B_s\otimes_R B_r$, that we will call (for reasons that will be clear later) $B_{sr}.$ Another example of  a Soergel bimodule for $S_3$ is the $\mathbb{Z}$-graded $R-$bimodule $B_{srs}:=R\otimes_{R^{s,r}}R(3)$.

We will use the convention that if $M$ and $N$ are two $R-$bimodules then their ``product'' is defined by $$MN:=M\otimes_R N.$$

We can now introduce the category of Soergel bimodules $\mathcal{B}(S_3)$ in our baby example $S_3$. They are $\mathbb{Z}$-graded $R-$bimodules that are isomorphic to direct sums and grading shifts of the following  set of $\mathbb{Z}$-graded $R-$bimodules
$$\II=\{R, B_s, B_r, B_{sr}, B_{rs}, B_{srs}\}.$$

\textbf{Philosphy: } \emph{One should think of $B_s$, $B_sB_r$ and $B_{srs}$ as analogous objects  to the elements $b_s$,  $b_{sr}$ and   $b_{srs}$ respectively in the Hecke algebra. One should also think of the product (resp. direct sum) between Soergel bimodules as an analogue of  product (resp. sum) in the Hecke algebra. Shifting the degree of a Soergel bimodule by one should be seen as multiplying the corresponding element in the Hecke algebra by $v$. We will make  this statement precise at some point.  }

This philosophy will gently emerge in the following pages. 
Recall that the Hecke algebra $\cH(S_3)$ is free over $\ZZ[v,v^{-1}]$ with basis $$\{1, b_s, b_r,b_{sr}, b_{rs},b_{srs}\}.$$  

\subsection{Stability of Soergel bimodules}

\subsubsection{The crucial phenomena}\begin{st}\label{st}
The category $\mathcal{B}(S_3)$ is stable under product. 
\end{st}

\begin{proof}
It is obvious that we just need to prove that we can write any product of two elements of $\II$ as a direct sum of shifts of elements in $\II.$

One important fact about $R$ is that if $p\in R$, then $p-s\cdot p \in (y-x)R^s.$  For example, if $p=3xy^2z^7+ yz$, 
 \begin{align*}
p-s\cdot p &=3xy^2z^7+ yz -3yx^2z^7- xz \\
&= (y-x)(3xyz^7+z).
\end{align*}
 
 It is an easy exercise to convince oneself of this fact (hint: start with monomials). One can also see this fact more conceptually by noticing that the polynomial $p-s\cdot p$ vanishes in the hyperplane defined by the equation $y=x$. 
The same result stands for $r$. The element $p-r\cdot p$ is  $(z-y)$ multiplied by some element of $R^r.$ We define $\alpha_s:=y-x$ and $\alpha_r:=z-y$. If we define $$P_s(p)=\frac{p+s\cdot p}{2} \in R^s \  \ \mathrm{ and  } \ \  \del_s(p)=\frac{p-s\cdot p}{2\alpha_s} \in R^s$$
 then we have the decomposition \begin{equation}\label{p}p=P_s(p)+\alpha_s\del_s(p).\end{equation} This equality gives rise to an isomorphism of graded  $R^s$-bimodules
 \begin{equation}\label{Rgr}
 R\cong R^s\oplus R^s(-2).
 \end{equation}

As a direct consequence of this equation we obtain the isomorphism
\begin{eqnarray}\label{BsBs}
B_sB_s &\cong& R\otimes_{R^s}R\otimes_{R^s}R(2) \nonumber\\
&\cong&  R\otimes_{R^s}R(2)\,  \oplus \,  R\otimes_{R^s}R \nonumber\\
&=& B_s(1)\oplus B_s(-1)
\end{eqnarray}
Compare this isomorphism with the equality $b_sb_s=vb_s+v^{-1}b_s$ in the Hecke algebra. 
We also obtain the following isomorphism
\begin{eqnarray}\label{BsBsrs}
B_sB_{srs} &\cong& R\otimes_{R^s}R\otimes_{R^{s,r}}R(4) \nonumber\\
&\cong&  B_{srs}(1) \oplus  B_{srs}(-1).
\end{eqnarray}
Compare this isomorphism with the equality $b_sb_{srs}=vb_{srs}+v^{-1}b_{srs}$ in the Hecke algebra. 

Let us recall a classic result of invariant theory. There is an isomorphism of $R^{s,r}$-bimodules (see, for example  \cite[ch. IV, cor. 1.11 a.]{Hi})
\begin{equation}\label{sr} R\cong \bigoplus_{w\in S_3}R^{s,r}(-2l(w))\end{equation}

This isomorphism implies  that 
\begin{equation}\label{w0}
B_{srs}B_{srs}\cong B_{srs}(-3)\oplus B_{srs}(-1)^{\oplus 2}\oplus B_{srs}(1)^{\oplus 2}\oplus B_{srs}(3)
\end{equation}
Compare this isomorphism with the equality $$b_{srs}b_{srs}=(v^{-3}+2v^{-1}+2v^{1}+v^3)b_{srs}$$ in the Hecke algebra.

To finish the proof of the Baby Stability Theorem \ref{st} we just need to prove the following isomorphism (that one should compare with the formula  in the Hecke algebra $b_sb_rb_s=b_{srs}+b_s).$
\subsection{$B_sB_rB_s\cong B_{srs}\oplus B_s$}\label{dec}

 For this we need first to define four morphisms of $R$-bimodules. The first one is the multiplication morphism $m_s\in \mathrm{Hom}(B_s, R)$
\begin{align*}
m_s: R\otimes_{R^s}R(1)&\rightarrow R\\
p\otimes q &\mapsto pq
\end{align*}
that is obviously a (degree $1$) morphism. The second morphism  $m_s^a\in \mathrm{Hom}(R, B_s)$ is its adjoint, in a sense that will become clear in Section \ref{adjunction} (thus explaining the notation  used for this morphism).
\begin{align*}
m_s^a: R &\rightarrow R\otimes_{R^s}R(1)\\
1 &\mapsto \alpha_s\otimes 1+ 1\otimes \alpha_s
\end{align*}

To check that this is  a (degree $1$) morphism of $R$-bimodules we need to check that for any $p\in R$ we have  $m_s^a(1)p=pm_s^a(1).$ This is
\begin{align*}
\alpha_s\otimes p+ 1\otimes \alpha_sp&=\alpha_sP_s(p)\otimes1+\alpha_s\del_s(p)\otimes\alpha_s+P_s(p)\otimes\alpha_s +\del_s(p)\alpha_s^2\otimes1\\
 &= p\alpha_s\otimes 1+p\otimes \alpha_s.
\end{align*}
For the first equality we used equation (\ref{p}) and the fact that $$\alpha_s^2=(x+y)^2-4xy\in R^s.$$
The following  morphism $j_s\in \mathrm{Hom}(B_sB_s, B_s)$ has  degree $-1$
\begin{align*}
j_s: R\otimes_{R^s}R\otimes_{R^s}R(2)&\rightarrow R\otimes_{R^s}R(1)\\
p\otimes q \otimes h&\mapsto p\del_s(q)\otimes h
\end{align*}
If it is a well defined map it is obvious that it is an $R$-bimodule morphism. So one just needs to check that $j_s(pr^s\otimes q \otimes h)=j_s(p\otimes r^s q \otimes h)$ and that $j_s(p\otimes q r^s\otimes h)=j_s(p\otimes  q \otimes r^sh)$, for any $r^s\in R^s.$ Both equations follow from the fact that $\del_s$ is a morphism of left (or right) $R^s$-modules (this is easy to check). 

And finally, the last morphism, the adjoint of $j_s$ (also of degree $-1$)
\begin{align*}
j_s^a: R\otimes_{R^s}R(1)  & \rightarrow   R\otimes_{R^s}R\otimes_{R^s}R(2)\\
p\otimes q & \mapsto p\otimes 1\otimes q
\end{align*}

\begin{nota}\label{cucu}
When it is clear from the context we will not write the identity morphisms. For example if we write $m_r$ for a morphism in $\mathrm{Hom}(B_sB_rB_s,B_sB_s)$ we mean $\mathrm{id}\otimes m_r\otimes \mathrm{id}.$ 
\end{nota}

\begin{ef}
The morphism $e:=   - m_r^a \circ j_s^a\circ j_s\circ m_r\in \mathrm{End}(B_sB_rB_s)$ is an idempotent (here we are using Notation \ref{cucu}).
\end{ef}
\begin{proof}
It is enough to check that $j_s\circ m_r\circ m_r^a \circ j_s^a=-\mathrm{id}\in \mathrm{End}(B_s),$ which is trivial. 
\end{proof}
From this fact we deduce that 
\begin{equation}\label{e}B_sB_rB_s=\mathrm{im}(1-e)\oplus \mathrm{im}(e),\end{equation} because if $e$ is an idempotent $1-e$ is easily checked to be an idempotent orthogonal to $e$. It is obvious that $m_r$ and $j_s$ are surjective and that $m_r^a$ and $j_s^a$ are injective morphisms. This implies that $$\mathrm{im}(e)\cong \mathrm{im}(m_r^a \circ j_s^a)\cong B_s.$$

To finish the proof we need to check the following isomorphism of graded $R$-bimodules
\subsection{$\mathrm{im}(1-e)\cong B_{srs}.$}\label{achi}
 
Let $1^{\otimes}:=1\otimes 1\otimes 1\otimes 1\in R\otimes_{R^s}R\otimes_{R^r}R\otimes_{R^s}R$
and let us denote  by $\langle 1^{\otimes} \rangle$   the $R$-bimodule generated by    $1^{\otimes}$.
We will prove the Lemma in two steps.

\subsubsection{Step 1}\ \textit{ We will prove that }$\mathrm{im}(1-e)= \langle 1^{\otimes} \rangle$.

As $(1-e)(1^{\otimes})=1^{\otimes}$ we have that $$\langle 1^{\otimes} \rangle  \subseteq    \mathrm{im}(1-e).$$
 It is a fun and easy exercice (using twice equation (\ref{p}) and some smart juggling with the variables) to see that  $B_sB_rB_s$ is generated as an $R$-bimodule by the two elements $1^{\otimes}$ and $1\otimes x\otimes 1\otimes 1$. We already know that $(1-e)(1^{\otimes})\in   \langle 1^{\otimes} \rangle$. We will do explicitly some of the juggling we said before to see that $ (1-e)(1\otimes x\otimes 1\otimes 1) \in\langle  1^{\otimes} \rangle$.

Firstly we note that in $B_sB_rB_s$ we have \begin{equation}\label{eq1}1\otimes y\otimes 1\otimes1 =(x+y)\otimes 1\otimes 1\otimes 1-1\otimes x\otimes 1\otimes 1\end{equation}
By definition we have
$$
(1-e)(1\otimes x\otimes 1\otimes 1)= 1\otimes x\otimes 1\otimes 1-\frac{1}{2}\otimes (z-y)\otimes 1\otimes 1 -\frac{1}{2}\otimes 1\otimes (z-y)\otimes 1$$
We apply equation (\ref{eq1}) in the second and third terms of the right hand side  and we obtain 
$$
(1-e)(1\otimes x\otimes 1\otimes 1)= \frac{1}{2} (x+y-z)\otimes 1\otimes 1\otimes 1 +1\otimes 1\otimes 1\otimes \frac{1}{2}(x+y-z).$$
Thus we conclude Step 1 of the proof.

\subsubsection{Step 2}\ \textit{ We will prove that }$ \langle 1^{\otimes} \rangle \cong B_{srs}$.

 It is clear that 
\begin{align*}
 R\otimes_{R^{s,r}}R(3)  & \rightarrow  B_sB_rB_s\\
p\otimes q & \mapsto p\otimes 1\otimes 1\otimes q
\end{align*}
is a graded $R$-bimodule morphism. It is also clear that the image is $  \langle 1^{\otimes} \rangle$. Because of the isomorphism (\ref{sr}) we have that 
$$B_{srs} \cong  R(-3)\oplus R(-1)^{\oplus 2}\oplus R(1)^{\oplus 2}\oplus R(3)$$
 as a graded left $R$-module. We know by (\ref{Rgr}) that $$B_s\cong R(-1)\oplus R(1)$$ as a graded  left-$R$ module and also 
\begin{align*}
B_sB_rB_s&\cong (R(-1)\oplus R(1))\otimes_R(R(-1)\oplus R(1))\otimes_R (R(-1)\oplus R(1))\\
&\cong R(-3)\oplus R(-1)^{\oplus 3}\oplus R(1)^{\oplus 3}\oplus R(3)
\end{align*}
By  (\ref{e}) we see that in each graded degree $B_{srs}$ and $\mathrm{im}(1-e)$ have the same dimension as finite dimensional $\RR$-vector spaces. As a surjective map between isomorphic vector spaces is an isomorphism, we conclude the proof of Lemma \ref{achi}, Proposition \ref{dec} and Theorem \ref{st}. 
\end{proof}

\subsection{Soergel's categorification Theorem}

\subsubsection{Indecomposables}
Why did we us the letter $\II$ to denote the set $\{R, B_s, B_r, B_{sr}, B_{rs}, B_{srs}\}$? Because they are indecomposable objects. Let us see why.

We will use the following notation. A $\mathrm{min}$ in the subindex of a graded object means the minimal $i$ for which its degree $i$ part is non-zero. For example $(B_s)_{\mathrm{min}}=(B_s)_{-1}$

As we have seen, every  element $M\in \II$ is generated as an $R$-bimodule by the element $1^{\otimes}\in M_{\mathrm{min}}$. Let us suppose that $M=N\oplus P$.  Then we have that $M_{\mathrm{min}}=N_{\mathrm{min}}\oplus P_{\mathrm{min}}$ as $\RR$-vector spaces, but the dimension of $M_{\mathrm{min}}$ over $\RR$ is one, so $M_{\mathrm{min}}$ is either $N_{\mathrm{min}}$ or $P_{\mathrm{min}}$ and thus $M$ is either $N$ or $P$. Thus $M$ is indecomposable. 

\textbf{Caution!}  \emph{It is a particularity of the  group $S_3$ that all the indecomposable Soergel bimodules are generated  by $1^{\otimes}$. For example in $S_4$ the indecomposable $B_{s_2s_1s_3s_2}$ is not generated by $1^{\otimes}$. }

\subsubsection{How to produce an algebra from Soergel bimodules}

Let us introduce the 
\begin{nota}\label{poly}
If $p=\sum_i a_i v^i \in \NN[v,v^{-1}]$ and $M$ is a graded bimodule,  we will denote by $p\cdot M$ the graded bimodule $$\bigoplus_{i\in \ZZ}M(i)^{\oplus a_i}$$
\end{nota}

One can see that Soergel bimodules encode all the information of the Hecke algebra. Let us try to make this idea more precise. Consider the following ``algebra'' $\NN\cB(S_3)$ over $ \NN[v,v^{-1}]$ (this is not a ring because it lacks of additive inverse, so $\NN\cB(S_3)$ is strictly speaking not an algebra, but apart from this ``detail'' it satisfies all the other defining properties of an algebra): the elements of $\NN\cB(S_3)$ are Soergel bimodules modulo isomorphism. We denote by $\langle M\rangle$ the isomorphism class of $M$. In this algebra, the sum is defined to be direct sum $\langle M\rangle+\langle N\rangle:=\langle M\oplus N\rangle$, the product is defined to be tensor product $\langle M\rangle \cdot \langle N\rangle:=\langle M N\rangle$ and multiplication by $v$ is shifting your graded bimodule by $1$, i.e. $v\cdot \langle M\rangle:= \langle M(1)\rangle$. 

\begin{prop}\label{iso}
There is an isomorphism of ``algebras''
\begin{align*}\NN\cB(S_3)&\cong \bigoplus_{w\in S_3}\NN[v,v^{-1}]b_w\subset \cH =\bigoplus_{w\in S_3}\ZZ[v,v^{-1}]b_w\\ 
\langle B_w\rangle &\mapsto b_w.
\end{align*}
\end{prop}
\begin{proof}

The  equations (\ref{BsBs}), (\ref{BsBsrs}),  (\ref{sr}),  (\ref{w0}) and  Proposition \ref{dec} tell us that the multiplication rules in both algebras are equal. 
The only problem that might appear is that $\NN\cB(S_3)$ might not be free as a $\NN[v,v^{-1}]$- module over the set $\langle B_w\rangle_{w\in S_3}$, i.e. one might have that $$\sum_{w\in S_3}p_w \langle B_w\rangle\cong \sum_{w\in S_3}q_w \langle B_w\rangle $$ with $p_w, q_w\in \NN[v,v^{-1}]$. But this is not possible because we have that the category of $\ZZ$-graded finitely generated $R$-bimodules  (or, what is the same, $R\otimes_{\RR}R$-modules) admits the Krull-Schmidt Theorem. For a proof see \cite[Section 5.4]{Pi}.
\end{proof}

\subsubsection{Recovering the Hecke algebra}

If we want to produce  an algebra isomorphic to $\cH$ starting with the category of Soergel bimodules and not just ``the positive part'', we just need to add formally to $\NN\cB(S_3)$ the element $-\langle M\rangle$ for every  Soergel bimodule $M$.  This element will satisfy the equation $-\langle M\rangle+\langle M\rangle=0.$ In this manner we obtain a honest algebra over $ \ZZ[v,v^{-1}]$, isomorphic to $\cH.$
 One formal way of doing this is  with the following general definition. 
\begin{defi} Let $\cA$ be an additive category. The \emph{split Grothendieck group} of $\cA$ denoted by $\langle \cA\rangle$ is the free abelian group over the objects modulo the relations $M=N+P$ whenever we have $M\cong N\oplus P.$ Given an object $A\in \cA$, let $\langle A\rangle$ denote its class in $\langle \cA\rangle$.
\end{defi}

In the case of $\langle \cB(S_3)\rangle$, this group can be endowed with a structure of $ \ZZ[v,v^{-1}]$-algebra, as we have seen (addition is direct sum, product is tensor product, etc). 

\begin{sca} If  $W$ is $S_3$ we have
\begin{itemize}
\item  The set $W\times \ZZ$ is in bijection with the set of indecomposable Soergel bimodules via the map $$(w,m)\mapsto B_w(m)$$
\item The map \begin{align*}\langle \cB(S_3)\rangle& \rightarrow \cH(S_3)\\ \langle B_w\rangle &\mapsto b_w\\ \langle R(1)\rangle&\mapsto v \end{align*} is an isomorphism of $\ZZ[v,v^{-1}]$-algebras. 
\end{itemize}
\end{sca}

\begin{proof}
The first part has already been proved. The second one also, modulo the remark that before we used the notation $\langle N\rangle =\langle M\rangle$ if $N$ and $M$ are isomorphic, and now we are using it for two elements equal in the Grothendieck group and this could be confusing notation. In fact, due to the Krull-Schmidt property explained in the proof of \ref{iso} these two notations mean the same thing. 
\end{proof}

The morphisms between Bott-Samelson bimodules (i.e. bimodules of the form $B_sB_r\cdots B_t$ for $s,r, \ldots, t\in S$) in principle could be quite complicated, or even atrocious, but we are, oh so very lucky. Two miracles happen. Firstly the Hom spaces are free as $R$-modules. This  is highly non-trivial. 

The second miracle is that there is a combinatorial set (defined by the author in \cite{Li1}) in the Hom spaces, called ``Light leaves'' that is a basis of this free Hom space. In the next section we will introduce Soergel bimodules for any Coxeter group and we will explain  the construction of light leaves before we  can calculate the indecomposable Soergel bimodules for the examples B, C and D in Section \ref{final}.

\section{Soergel bimodules and light leaves in ranks 1 and 2}\label{mor}

We will present the general definition of the category of Soergel bimodules, but we will still work over the field of real numbers. Over a field of positive characteristic  the categorification theorem still works, but Soergel bimodules behave quite differently (projectors from the Bott-Samelsons are not the same as over $\RR$). 

We will recall most of  the definitions given  in the last section to make this section independent of the last one.

\subsection{Soergel category $\cB$ for any Coxeter system over $\RR$}\label{any}
Let $(W,S)$ be an arbitrary Coxeter system. Consider $V=\oplus_{s\in S} \mathbb{R}e_s$ the Geometric Representation. It is a linear representation defined by the formula 
$$s\cdot e_r=  e_r+2\, \mathrm{cos}\, \biggl(\frac{\pi}{m_{sr}}\biggr) e_s \hspace{1cm}\mathrm{for\ all\ }s, r \in S,$$
where $m_{sr}$ is the order of the element $sr$ in $W.$ By convention $\pi/\infty =0.$

Let $R=R(V)$\label{d1}  be the algebra of regular functions on $V$ with  the grading induced by putting $V^*$ in degree two, i.e. $R=\bigoplus_{i\in
\mathbb{Z}}R_i$ with  $R_2 = V^*$ and $R_i=0$ if $i$ is odd. The 
action of $W$ on $V$ induces an action on $R$.

For any $\mathbb{Z}$-graded object $M=\bigoplus_i M_i,$ and every  $n\in \mathbb{Z}$, we denote by $M(n)$ the \emph{shifted object} defined by the formula $$(M(n))_i=M_{i+n}.$$

 With this notation in hand we can define, for $s\in S,$  the  $\mathbb{Z}-$graded  $R-$bimodule $$B_s=R\otimes_{R^s} R(1),$$ where $R^s$ is the subspace of $R$ fixed by $s$.   
Given $M,N\in \mathcal{B}$ we denote  their tensor product simply by juxtaposition: $M N := M \otimes_R N$.

If $ \underline{s}=(s_1,\ldots, s_n)\in {S}^n,$ we will denote by $B_{\underline{s}}$ the  $\mathbb{Z}-$graded $R-$bimodule $${B}_{s_1}{B}_{s_2}\cdots {B}_{s_n}\cong {R}\otimes_{{R}^{s_1}}{R}\otimes_{{R}^{s_2}}\cdots \otimes_{{R}^{s_n}}{R}(n).$$ We use the convention $B_{(\mathrm{id})}=R.$ Bimodules of the type $B_{\underline{s}}$ are called  \emph{Bott-Samelson bimodules.}

 The category of  \textit{Soergel bimodules} $\mathcal{B}=\mathcal{B}(W,S)$ is the full sub-category  of $\mathbb{Z}-$graded $R-$bimodules, with objects the shifts of finite direct sums of direct summands of Bott-Samelson bimodules.

For every essentially small additive category $\mathcal{A}$, we call  $\langle\mathcal{A}\rangle$ its \textit{split Grothendieck 
group}. It is the free abelian group generated by the objects of  $\mathcal{A}$ modulo the relations $M=M'+M''$ whenever we have $M\cong M'\oplus M''$. Given an object $A\in \mathcal{A},$ let  $\langle A \rangle$ denote its class in $ \langle \mathcal{A}
\rangle$.

In \cite{So0} Soergel proves \emph{Soergel's categorification theorem} (the version of all these results for the geometric representation explained here is proved in \cite{li2}), which consist of two statements. Firstly, that there exist a unique ring isomorphism, the \emph{character map}
$\mathrm{ch}:  \langle \mathcal{B}\rangle \rightarrow \mathcal{H},$
such that $\mathrm{ch}(\langle R(1)\rangle)=v$ and $\mathrm{ch}(\langle B_s \rangle)=(h_s+v)$.
Secondly, there is a natural bijection between the set of indecomposable Soergel bimodules and the set $W\times \mathbb{Z}.$ We call $B_x$ the indecomposable Soergel bimodule corresponding to $(x,0)$ under this identification.

\begin{ques} Define an analogue of Soergel bimodules for  complex reflection groups.
\end{ques}

Soergel also proves that if $sr\cdots t$ is a reduced expression for $x\in W$ then 
one has (recall Notation \ref{poly})$$ B_sB_r\cdots B_t\cong B_x\oplus \bigoplus_{y<x}q_y\cdot  B_y\hspace{.3cm}\mathrm{with }\ q_y\in \NN[v, v^{-1}].  $$

This formula, plus the fact that $B_x$ does not appear in the decomposition of any other Bott-Samelson of lesser length, gives a unique characterization of $B_x$. The following theorem, conjectured by Soergel in the early nineties is amazingly beautiful and it is probably the most powerful result in the theory. We will come back to it in the follow-ups of this paper. 

\begin{thm}[Elias and Williamson \cite{EW}]
$\mathrm{ch}(\langle B_x \rangle)=b_x.$
\end{thm}

\begin{remark}
It is is a result of Soergel that it is enough to prove that for each $x\in W$ there is  $M_x\in \mathcal{B}$ with $\mathrm{ch}(\langle M_x \rangle)=b_x.$ 
\end{remark}

We will represent morphisms between Bott-Samelson bimodules by drawing them in a very specific way (in subsequent papers of this saga we will go deeply in the reasons of why do we draw the morphisms in such a manner).  Let us start with ``one color morphisms''.

\subsection{Drawing morphisms: one color}

We fix a simple reflection $s\in S.$  We will start by explaining how to represent  in a drawing a morphism between ``one color Bott-Samelson bimodules'', i.e. bimodules having only $s$ in its expression, for example $B_sB_sB_sB_s.$ 

Morphisms will be drawn inside the strip $\RR \times [0,1]\subset \RR^2.$ This will be done in a bottom-up way, i.e. in the line $\RR \times \{0\}$ we will draw the same number of points as the number of $B_s$ that appear in the source of our morphism and in the line $\RR \times \{1\}$ we will draw the same number of points as the number of $B_s$ that appear in the target of our morphism (the bimodule $R$ is represented by the empty sequence). Here a list of examples were the lower black line is always $\RR \times \{0\}$ and the upper black line is $\RR \times \{1\}$ (recall the morphisms in \ref{dec}).

 \begin{figure}[H] 
\begin{center}
 \includegraphics[scale=0.25]{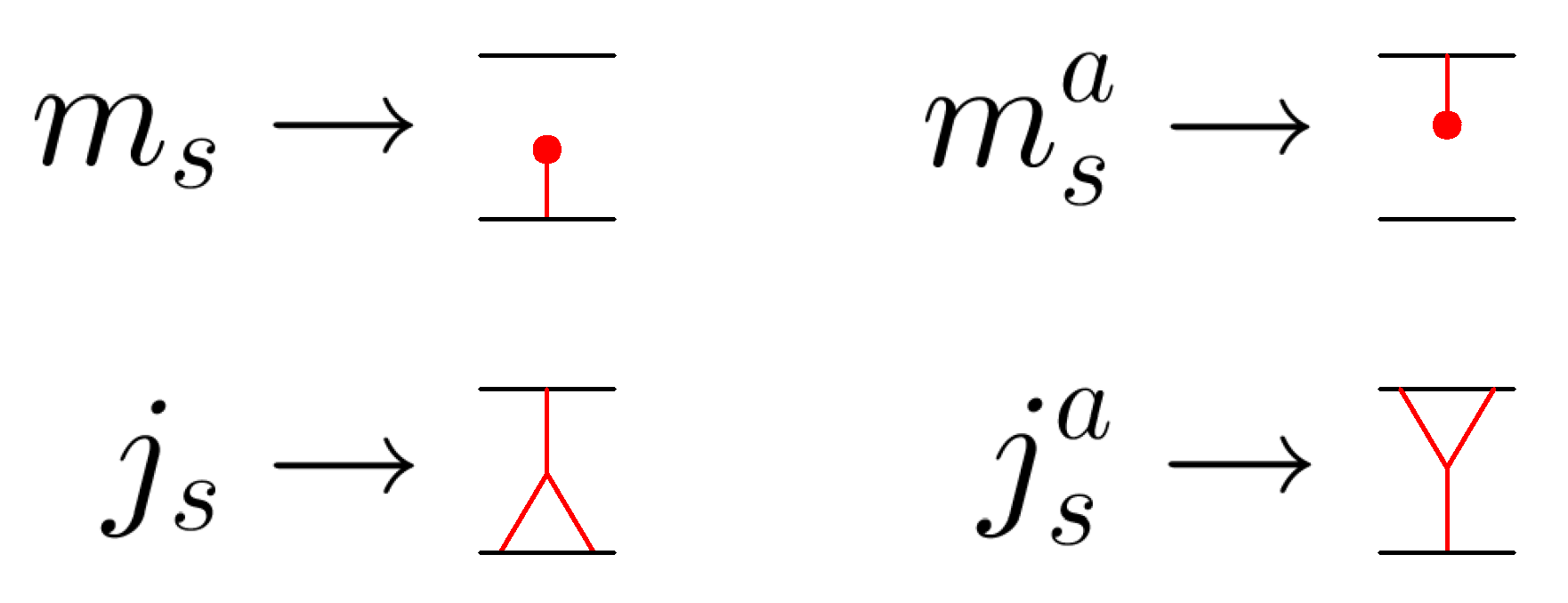} 
\caption{Important morphisms}  
\end{center}
\end{figure} 
We call $m_s$ and $m_s^a$ the \emph{dots} and $j_s$ and $j_s^a$ the \emph{trivalent vertices}. We will not longer draw the lines$\RR \times \{0\}$ and $\RR \times \{1\}$, but we will assume that they exist. 
A tensor product of morphisms is represented by glueing pictures, for example
 \begin{figure}[H] 
\begin{center}
 \includegraphics[scale=0.15]{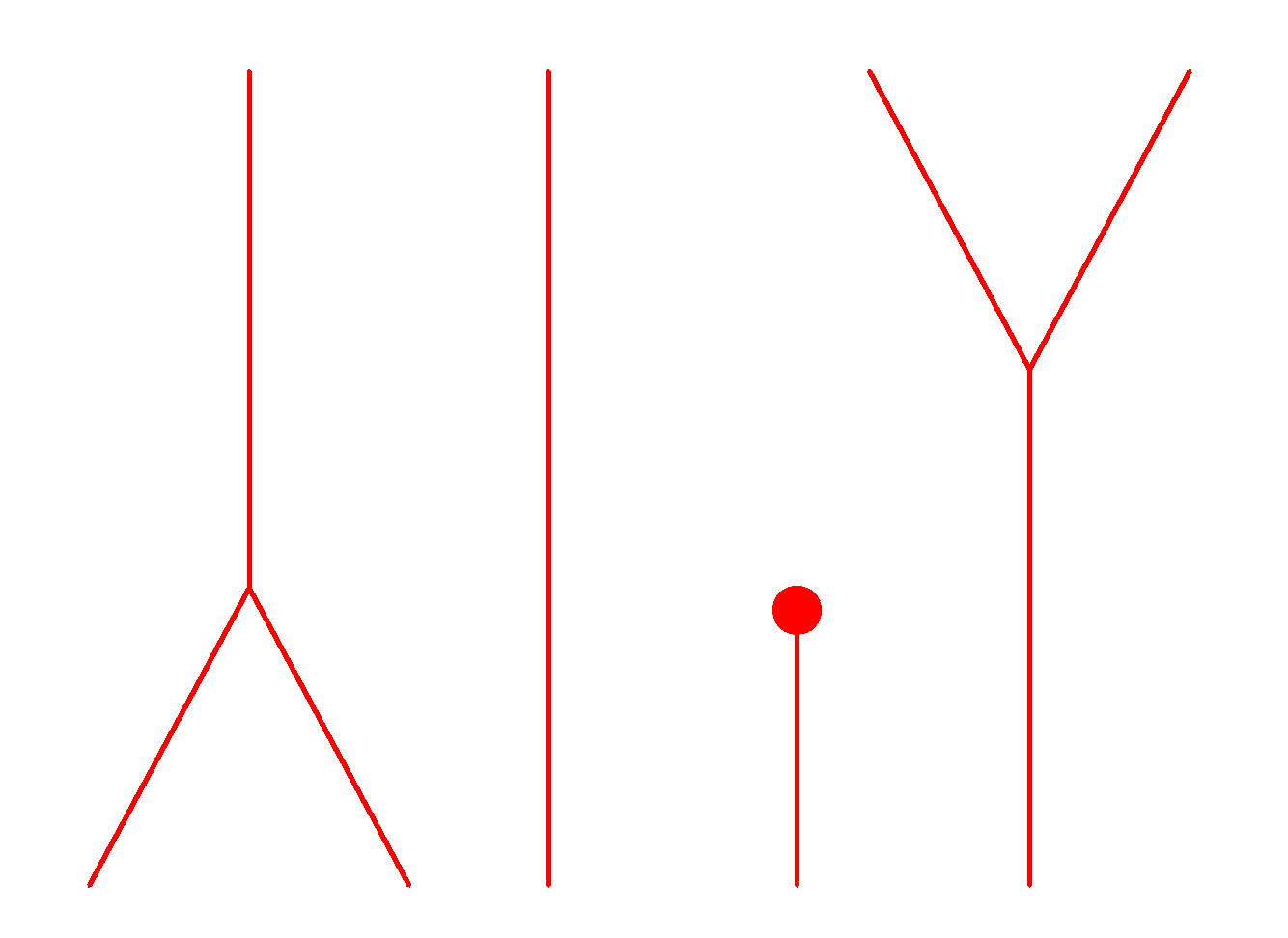} 
\caption{Tensor product}  
\end{center}
\end{figure} 
 represents the morphism 
$ j_s\otimes \mathrm{id}\otimes m_s\otimes j_s^a:B_sB_sB_sB_sB_s\rightarrow B_sB_sB_sB_s$.
Composition is represented by glueing  bottom-up in the obvious way. For example 
 \begin{figure}[H] \begin{center}
 \includegraphics[scale=0.2]{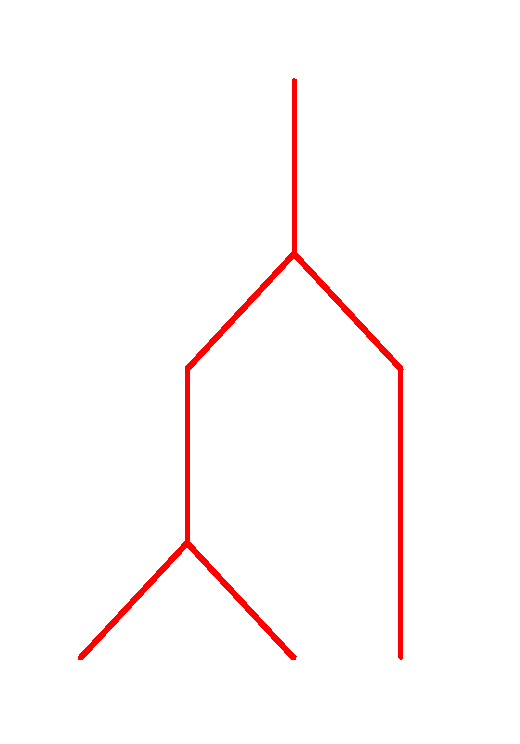} 
\caption{Composition}  
\end{center}
\end{figure} 
represents the morphism $j_s\circ (j_s\otimes \mathrm{id})$. 
 We introduce the following notation 
\begin{figure}[H] \begin{center}
 \includegraphics[scale=0.25]{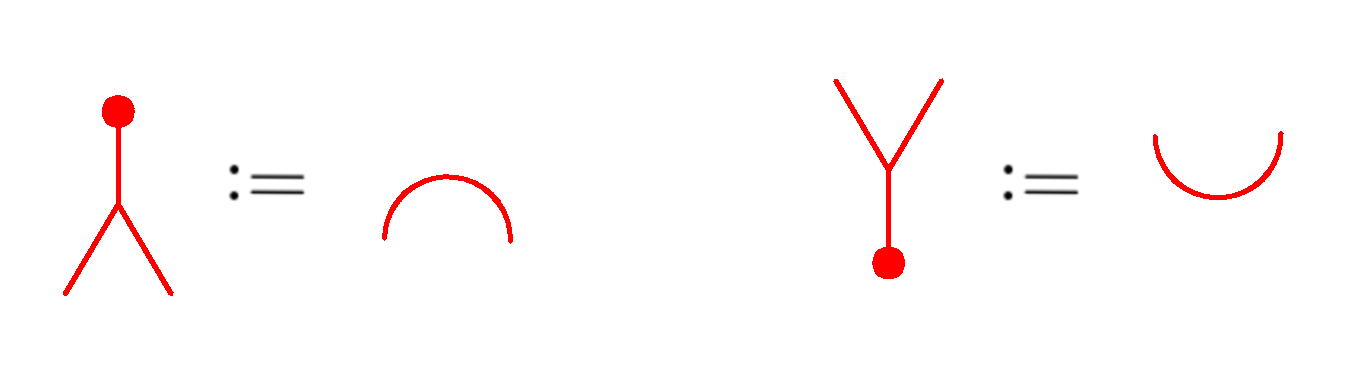} 
\caption{Cup and cap}  
\end{center}
\end{figure} 
The left-hand side morphism will be called the \emph{cup} and the right-hand side the \emph{cap}. It is easy to verify that $j_s\circ (j_s\otimes \mathrm{id})=j_s\circ ( \mathrm{id}\otimes j_s)$, or in pictures 
\begin{figure}[H] \begin{center}
 \includegraphics[scale=0.15]{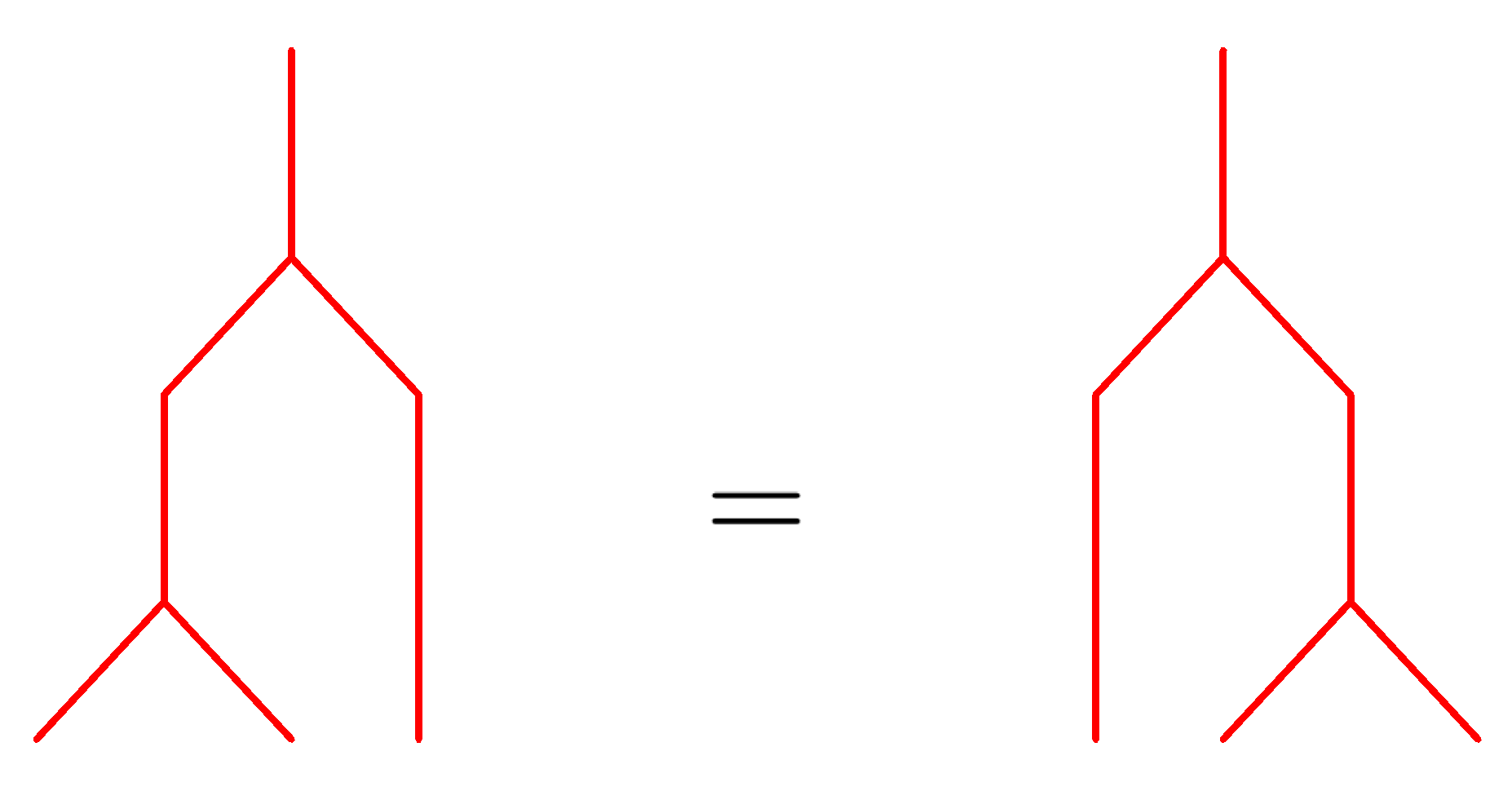} 
\caption{Associativity}  
\end{center}
\end{figure} 
This, of course means that any composition of $j_s$ tensored by identities in any order will give the same morphism, for example 
\begin{figure}[H] \begin{center}
 \includegraphics[scale=0.25]{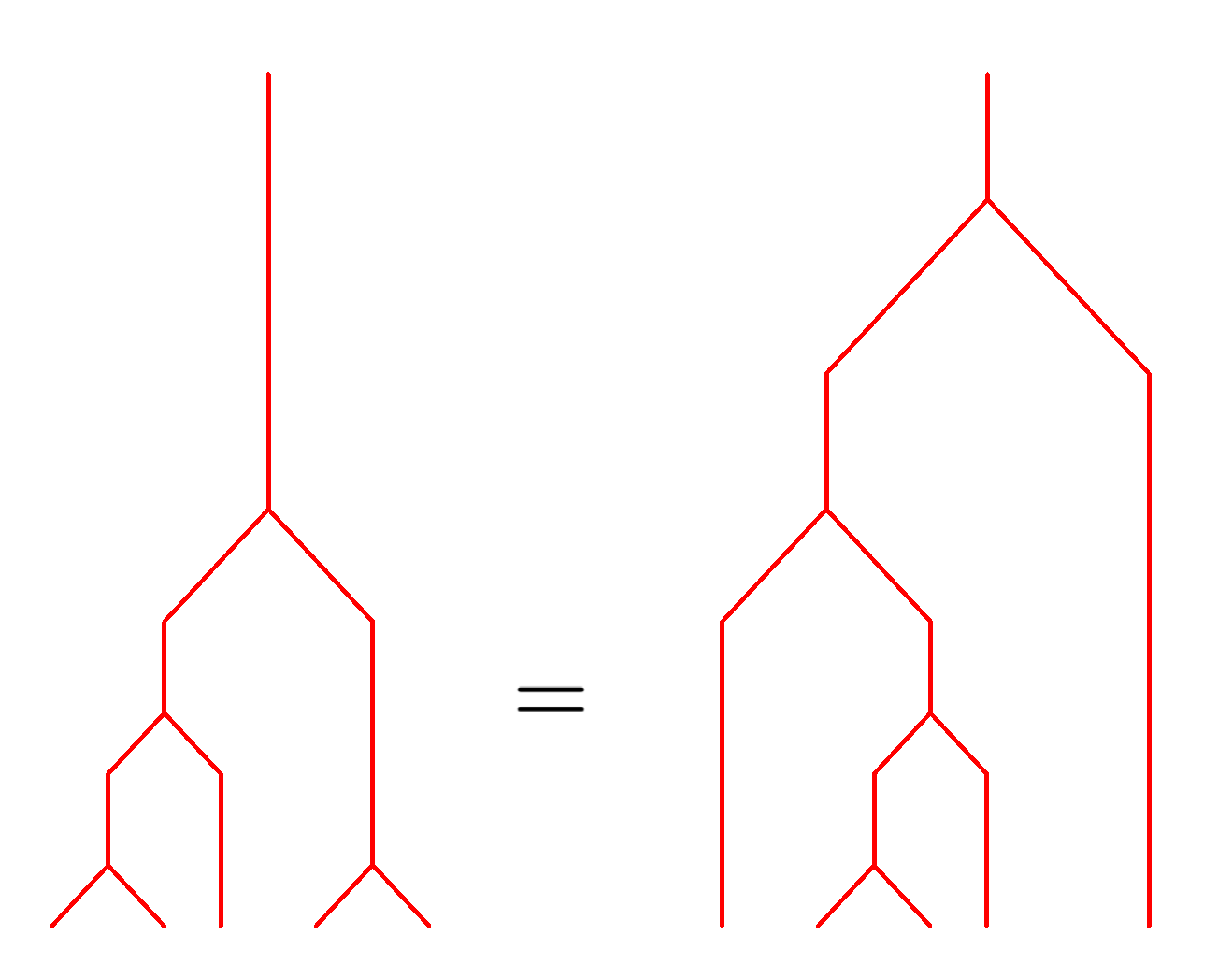} 
\end{center}
\end{figure} 
Given that there is no ambiguity and all these pictures represent the same morphism, we will denote this morphism by a comon picture that we call the \emph{hanging birdcage} \footnote{The ``birdcage'' terminology is non-standard and used for the first time in this paper.}
\begin{figure}[H] \begin{center}
 \includegraphics[scale=0.25]{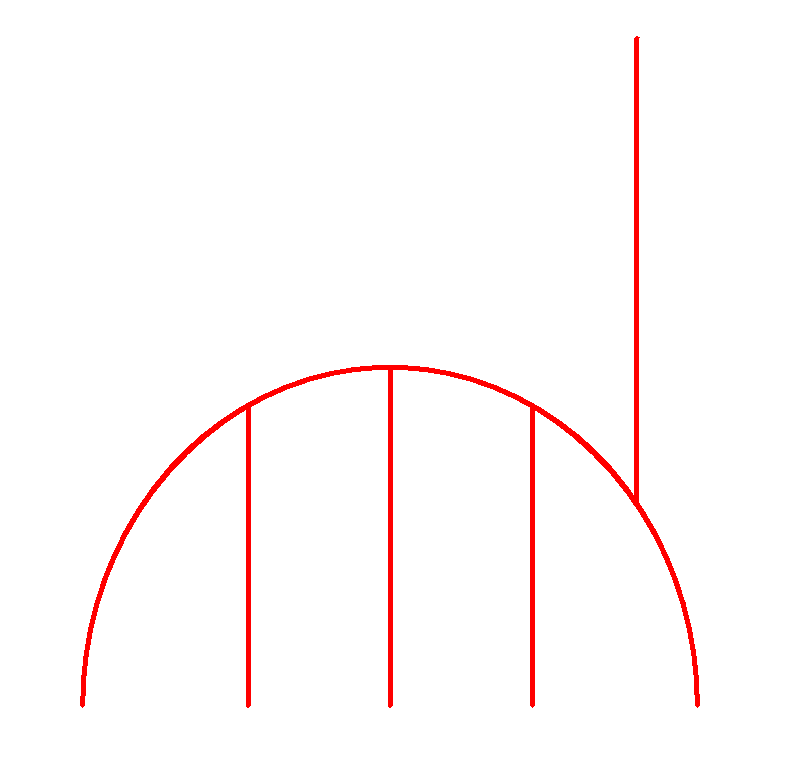} 
\caption{Hanging birdcage}  
\end{center}
\end{figure} 

The identity of $B_s$ (a vertical line) is also considered a hanging birdcage. 
If we compose this morphism with a dot we obtain the \emph{non-hanging birdcage} or simply, the \emph{birdcage.}
\begin{figure}[H] \begin{center}
 \includegraphics[scale=0.25]{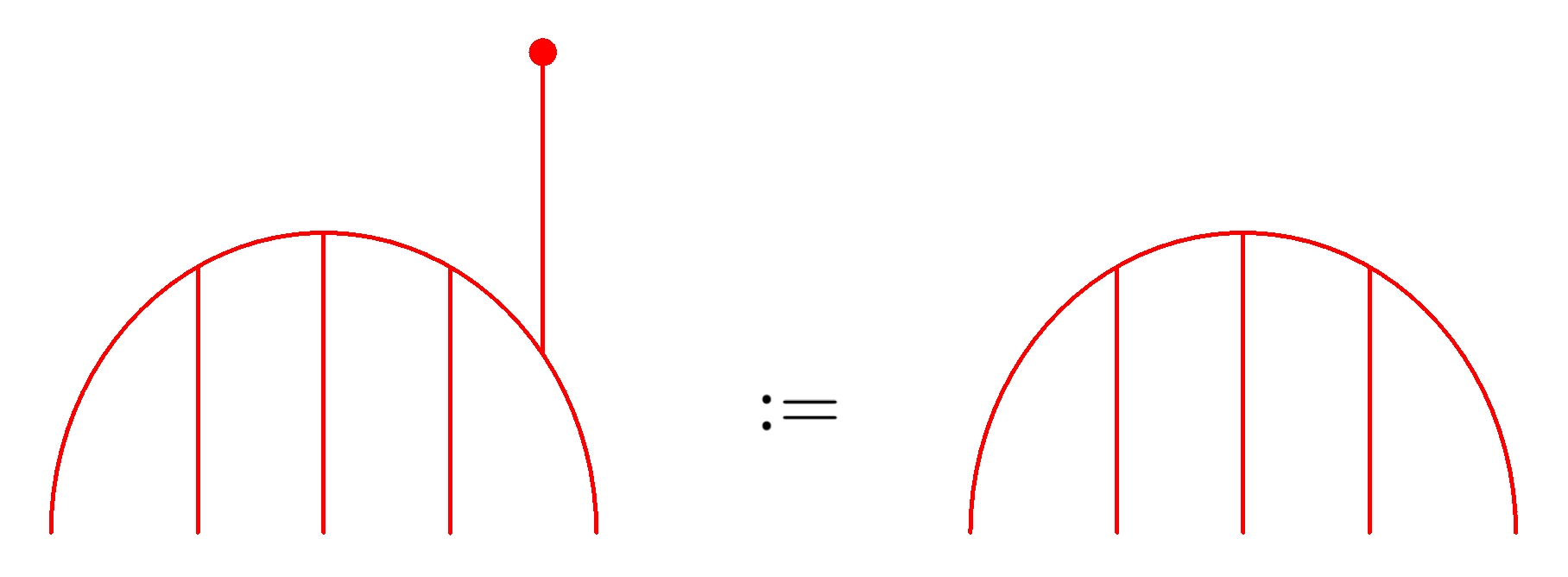} 
\end{center}
\caption{Birdcage}  
\end{figure} 
The \emph{length} of a birdcage is the number of $B_s$ that it has in its source. In the example, it is a length 5 birdcage. A dot is a length 1 birdcage. 

\subsection{One color light leaves}

A \emph{One Color Light Leaf} is a morphism built-up only with birdcages and dots. In the right-most position  one can admit  a hanging birdcage as in the picture.

\begin{figure}[H] \begin{center}
 \includegraphics[scale=0.28]{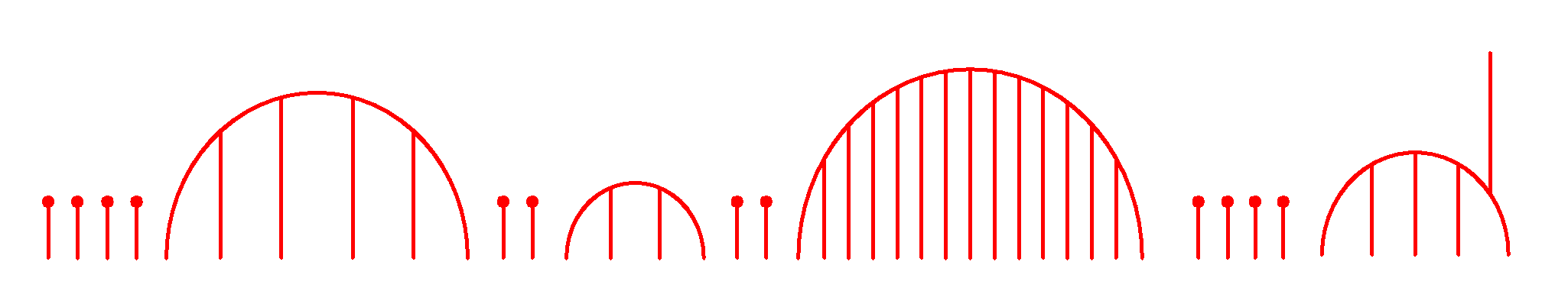} 
\end{center}
\caption{Example of a one-color light leaf}  
\label{LL1}
\end{figure}

\begin{figure}[H] \begin{center}
 \includegraphics[scale=0.25]{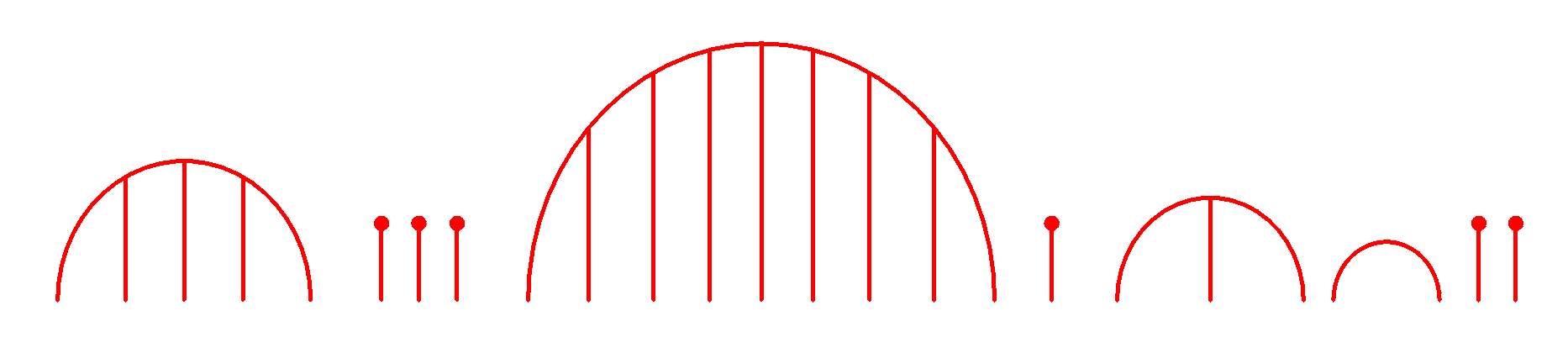} 
\end{center}
\caption{Another example}  
\label{LL2}
\end{figure} 

By definition a light leaf has two possible targets (as shown in figures \ref{LL1} and  \ref{LL2}). It can be $B_s$ if it ends with a hanging birdcage or $R$ if not. We will call $LL_n(s)$ the light leaves with source  $\underbrace{B_sB_sB_s\cdots}_{n \ \mathrm{times}}:=B_{s}^n$ and with target $B_s$, and $LL_n(e)$ the light leaves with the same source and target $R$. 

For any morphism $f$ represented by a  picture, one can flip the picture upside-down and thus obtain what we call the adjoint $f^a$, and this is again a well defined morphism with source and target flipped. For example, the adjoint of the morphism in Figure \ref{LL2} is the following 
\begin{figure}[H] \begin{center}
 \includegraphics[scale=0.25]{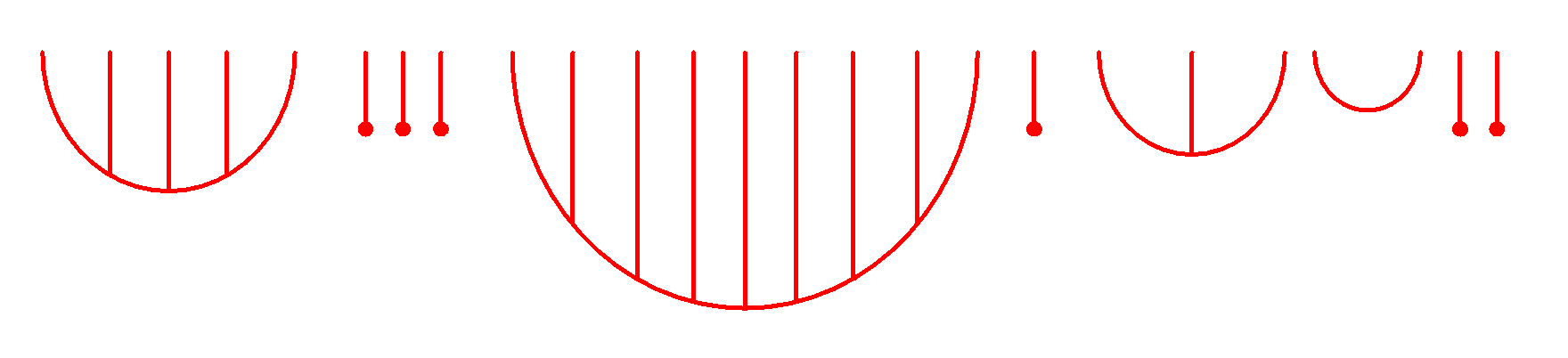} 
\end{center}
\end{figure}

Now we can state the ``one color'' Double Leaves Theorem. This is a particular case of a Theorem proved for any Coxeter system in \cite{li3}.
\begin{thm}\label{DL}
The set $$LL_m(s)^a\circ LL_n(s)\bigcup LL_m(e)^a\circ LL_n(e)$$ is a basis of $\mathrm{Hom}(B_s^n,B_s^m)$ as a right $R$-module. The elements of this basis are called \emph{double leaves.}
\end{thm}

\subsection{Two colors light leaves}

In this section we will consider only two simple reflections, say $s$ and $r$, represented by two colors, say red and blue. We will explain a version of the Double Leaves Theorem in two colors  but only for reduced expressions (we will see that  the general version of this theorem (Theorem \ref{DL})  does not restrict to reduced expressions neither in the source nor in the target). It is for simplicity of exposition that we use   reduced expressions. 

We will call a \emph{full birdcage} a morphism of the following type 
\begin{figure}[H] \begin{center}
 \includegraphics[scale=0.25]{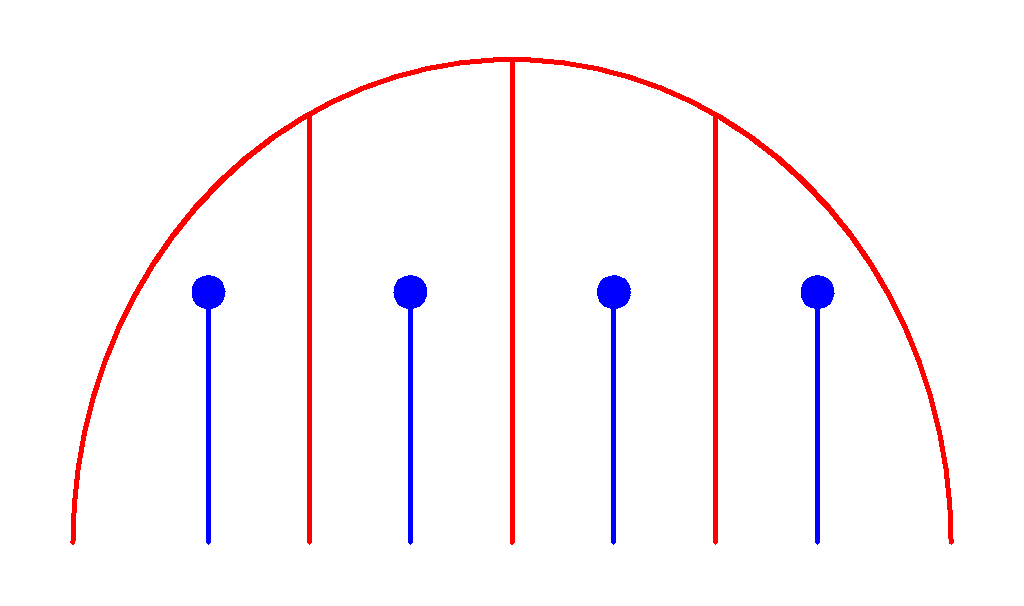} 
\end{center}
\caption{Full birdcage}  
\label{fbc}
\end{figure} 
from a bimodule of the form $B_sB_rB_s\cdots$ to $R$. We call it ``full'' by using the metaphor that it is full with ``birds'' (the dots), even though it is a bit sad to think this in such terms, poor birds. The \emph{color} of a full birdcage is the color of the birdcage, not the color of the birds. In Figure \ref{fbc} we have a red full birdcage.  A red dot will also be considered a red full birdcage (just think about the empty cage).

Let us play the following game. Start with a red dot. In each step of the game, each dot in our figure can be transformed into another full birdcage of the same color. For example, we transform the red dot into Figure  \ref{fbc}. Then we transform Figure \ref{fbc} into  
\begin{figure}[H] \begin{center}
 \includegraphics[scale=0.25]{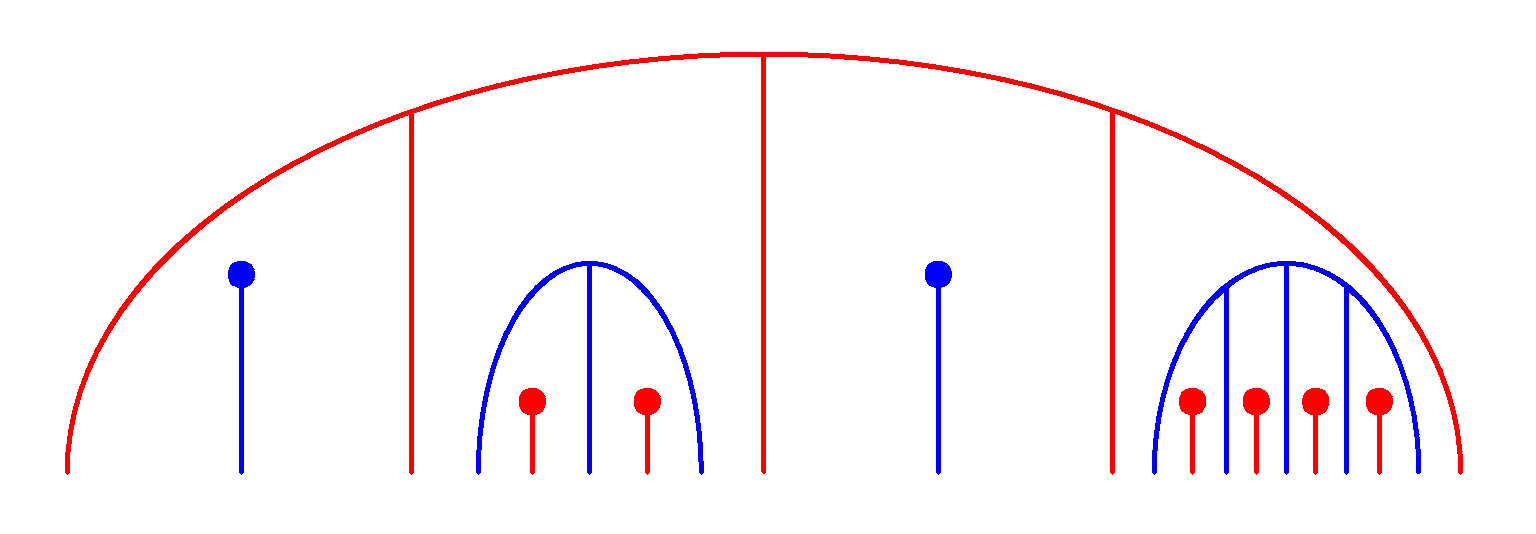} 
\end{center}
\caption{Third step in our game}  
\label{fbc2}
\end{figure} 
where we transformed the second and fourth blue dots into different full birdcages. Now we repeat this process
\begin{figure}[H] \begin{center}
 \includegraphics[scale=0.25]{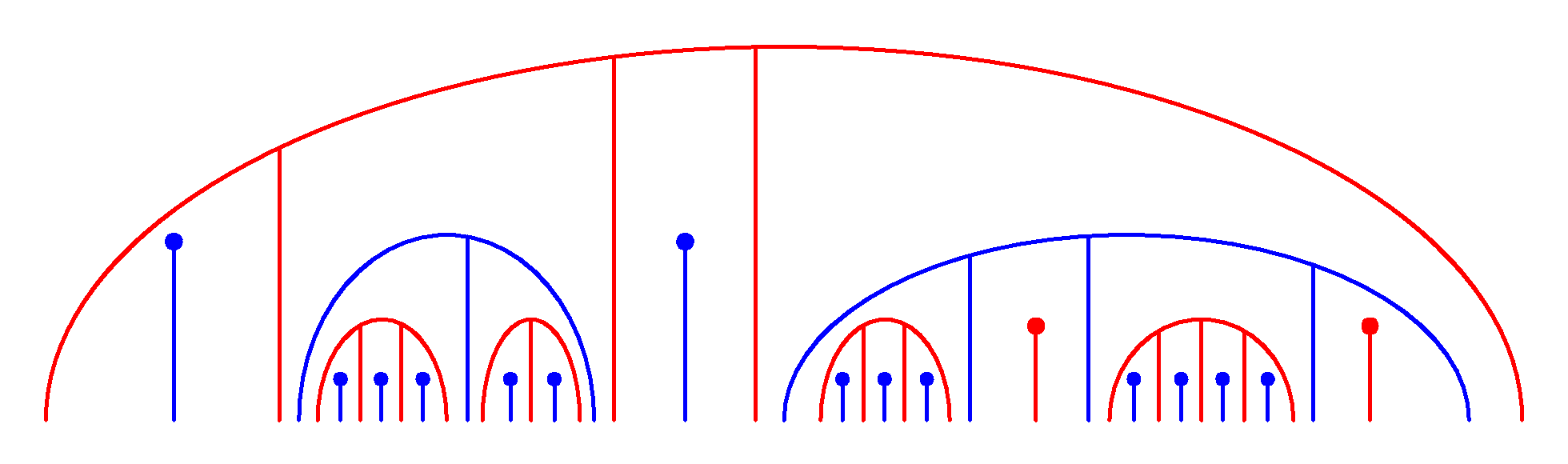} 
\end{center}
\caption{Fourth step in our game}  
\label{fbc3}
\end{figure} 
were we transformed the first, second, third and fifth red dots into different full birdcages. One can continue this process any number of times and every map one can obtain in this way will be called a \emph{birdcagecage}\footnote{Strictly speaking, the map in Figure \ref{fbc3} should be called birdcagecagecage, but it seems to us a bit impractical this name, so we stick with  birdcagecage.} (so in particular, any birdcage and any dot is a birdcagecage). If one adds a string as before, a birdcagecage is called a \emph{hanging birdcagecage} (again the degenerate case of a straight vertical line, the identity of some $B_s$, will also be considered a hanging birdcagecage).  

Let $W$ be a Dihedral group or a Universal Coxeter group in two generators (examples C and D).  If $\underline{x}= srs\cdots$ then we define the Bott-Samelson bimodule $B_{\underline{x}}:=B_sB_rB_s\cdots$. If we exclude the longest element for the Dihedral case then every element has only one reduced expression, so we can just call this bimodule $B_x$.

A  \emph{two-color light leaf} with source $B_x$ is a morphism built-up in three zones. The left zone (say zona A) is composed by birdcagecages. The middle zone (say zone B) is composed by hanging birdcagecages. The right zone (say zone C) is either empty or composed by just one birdcagecage. 
One example of a two color light leaf is
\begin{figure}[H] \begin{center}
 \includegraphics[scale=0.3]{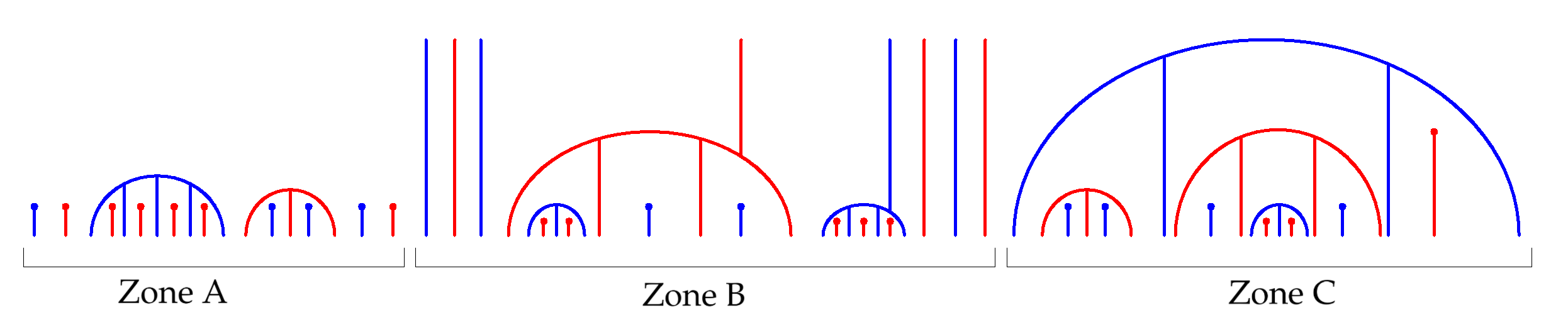} 
\end{center}
\caption{Example of a two-color light leaf}  
\label{tcll}
\end{figure} 

We will call $LL_{x}(z)$ the two color light leaves with source $B_x$ and target $B_z$. The following is a version of the Double Leaves Theorem in two colors 

\begin{thm}\label{DL}
Let $W$ be a finite or infinite Dihedral group and $x,y\in W$. If $W$ is finite, let us pick $x$ or $y$ different from $w_0$.
The set $$\bigcup_{z} LL_y(z)^a\circ LL_x(z)$$ is a basis of $\mathrm{Hom}(B_x,B_y)$ as a right $R$-module. 
\end{thm}

\section{Light leaves in general}\label{Light leaves basis}

\subsection{Adjunction}\label{adjunction}
Until now we have spoken about ``adjoint'' morphisms without really saying what the adjunction is. 

\begin{lem}
If $M$ and $N$ are two Bott-Samelson bimodules, the following map is an isomorphism of graded right $R$-modules 
\begin{align*}
\mathrm{Hom}(B_sM,N)&\rightarrow \mathrm{Hom}(M,B_sN)\\
f&\mapsto (\mathrm{id}_{B_s}\otimes f)\circ ((j_s^a\circ m_s^a)\otimes \mathrm{id}_M)
\end{align*}
\end{lem}

It is an exercice to prove this (one just needs to find explicitly the inverse map). 
 In any case, a detailed proof can be found in \cite[Lemme 2.4]{Li0}. This adjunction might seem complicated but in pictures, it is just this map
\begin{figure}[H] \begin{center}
 \includegraphics[scale=0.25]{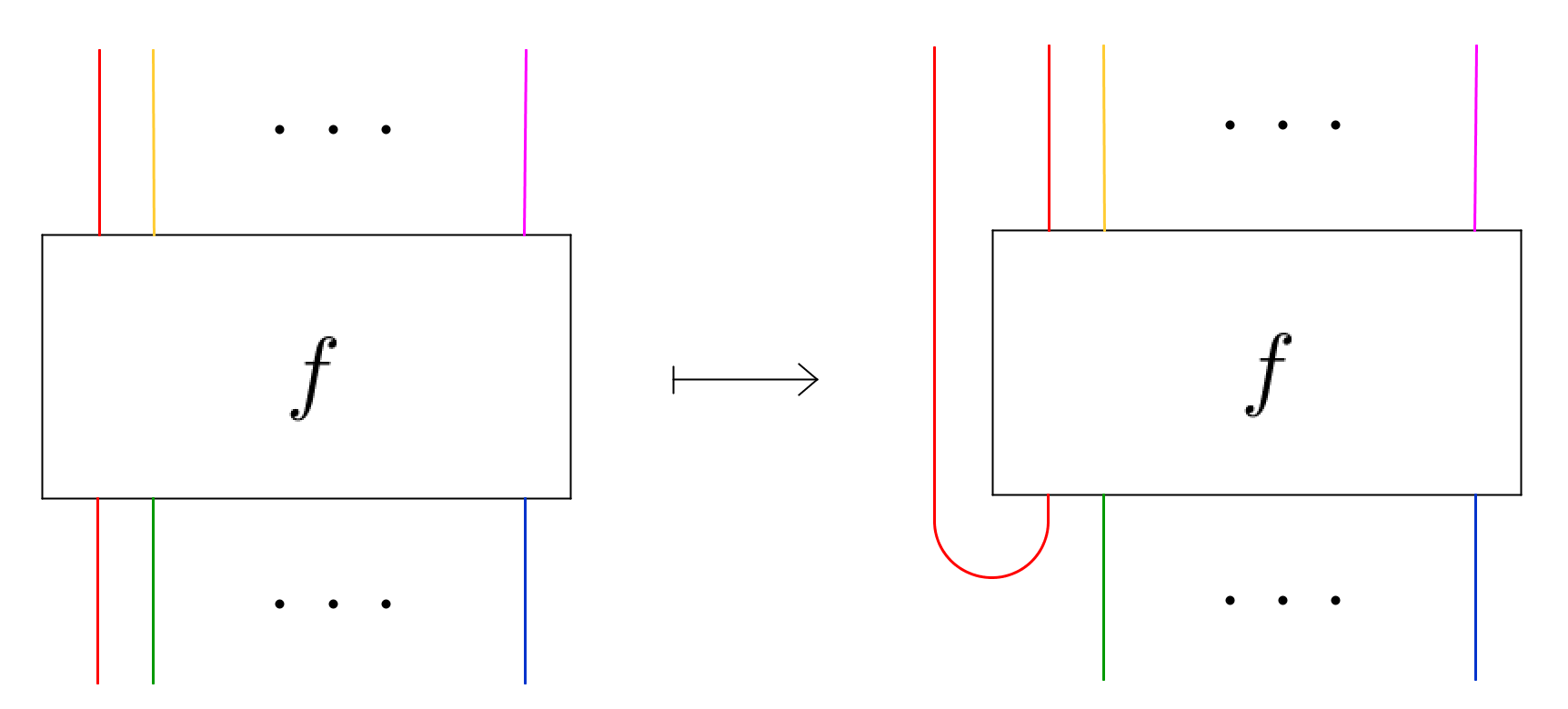} 
\caption{Adjunction map}  
\end{center}
\end{figure} 
It is an easy exercice to prove that $j_s^a$ is the adjoint of $j_s$ in this sense as it is $m_s^a$ of $m_s$. 
Before we can consrtuct the light leaves we need 

\subsection{A new morphism} Until now we have  just used the morphisms $m_s$, $j_s$ and its adjoints. We need a new morphism that does not appear in the examples that we have seen so far (it would appear in Theorem \ref{DL} if we would have considered $x$ or $y$ equal to the longest element, or if we would have considered as source a Bott-Samelson not represented by a reduced expression).

\subsubsection{The morphism $f_{sr}$}

Consider the bimodule $X_{sr}=B_sB_rB_s\cdots$ the product having $m(s,r)$ terms (recall that $m(s,r)$ is the order of $sr$). 
We define $f_{sr}$ as the only  degree $0$ morphism from $X_{sr}$ to $  X_{rs}$ sending $1^{\otimes}$ to $1^{\otimes}$. We have essentially encountered this morphism in Section \ref{babyA}.

Let $x=srs\cdots$ with the product having $m(s,r)$ terms. 
We have seen in Section \ref{any} that one has an isomorphism $$ X_{sr}\cong B_x\oplus \bigoplus_{y<x}p_y\cdot  B_y\hspace{.3cm}\mathrm{with }\ p_y\in \NN[v, v^{-1}]. $$
We have a similar formula for $X_{rs}$. The morphism $f_{sr}$ can be defined (modulo scalar) as the  projection from $X_{sr}$ onto $B_x$ composed with the inclusion of $B_x$ into $X_{rs}$.

 We have seen in the example $W=S_3$ that   $ B_{x}\cong R\otimes_{R^{s,r}}R.$ Moreover,  one can prove that this isomorphism is valid for any dihedral group, because the morphism
\begin{align*}
 R\otimes_{R^{s,r}}R&\rightarrow X_{sr}\\
1\otimes 1&\mapsto 1^{\otimes}
\end{align*}
is an isomorphism if one restricts the target to the sub-bimodule $$B_x=\langle 1^{\otimes}\rangle\subseteq X_{sr}$$ generated by $1^{\otimes}$. So to have an explicit formula for $f_{sr}$ we just need to find the inverse of this isomorphism. This is done in detail in \cite[Prop. 3.9]{li2.5}. By  construction it is clear that $f_{rs}$ is the adjoint of $f_{sr}.$ In \cite[Chapter 2]{Li0} one can find several interpretations of $f_{sr}$.

\subsubsection{Examples} Suppose we are in the case $W=S_n$. In this case $R=\RR[x_1,x_2, \ldots, x_n]$ and the group $W$ acts permuting the variables. We have that the simple reflections in $W$ are the $s_i$, the permutation switching the variables $i$ and $i+1$.  For simplicity of notation we will denote by $f_{ij}$ the morphism $f_{s_is_j}$ and by $B_i$ the bimodule $B_{s_i}$.
In this example we have three cases to consider. 
\begin{enumerate}
\item \textbf{First case:} $\vert i-j \vert>1$. The morphism $f_{ij}:B_iB_j\rightarrow B_jB_i$ is completeley determined by the formula $f_{ij}(1^{\otimes})=1^{\otimes},$ because the element $1^{\otimes}$ generates $B_iB_j$ as a bimodule. 
\item \textbf{Second case:} $j=i+1$. The morphism $f_{ij}:B_iB_{i+1}B_i\rightarrow B_{i+1}B_iB_{i+1}$ is completeley determined by the formulae $f_{ij}(1^{\otimes})=1^{\otimes}$ and $$f_{ij}(1\otimes x_i\otimes 1\otimes 1)=(x_i+x_{i+1})\otimes 1\otimes 1\otimes 1- 1\otimes 1\otimes 1\otimes x_{i+2}$$
\item \textbf{Third case:} $j=i-1$.  The morphism $f_{ij}:  B_{i+1}B_iB_{i+1}\rightarrow  B_iB_{i+1}B_i  $ is completeley determined by the formulae $f_{ij}(1^{\otimes})=1^{\otimes}$ and $$f_{ij}(1\otimes x_{i+2}\otimes 1\otimes 1)=1\otimes 1\otimes 1\otimes (x_{i+1}+x_{i+2})- x_i\otimes 1\otimes 1\otimes 1.$$
\end{enumerate}

\subsection{Path morphisms}\label{f}

\begin{defi} For $(W,S)$ a Coxeter system and $x\in W$ we define the \emph{Reduced expressions graph of $x$} or simply \emph{Rex($x$)} as the graph with nodes the set of reduced expressions of $x$ and  two reduced expressions are joined by an edge if they differ by a single braid relation. 
\end{defi}

For every reduced expression $\underline{s}$ we have associated a Bott-Samelson bimodule $ B_{\underline{s}}.$ If two Bott-Samelson bimodules $B, B'$ differ  by just one braid relation one has a morphism of the type $\mathrm{id}\otimes f_{sr}\otimes \mathrm{id}\in \mathrm{Hom}(B,B')$. For example, for the braid move $pq\,\, srs\,\, u\rightarrow pq\,\, rsr\,\, u$ (here we suppose $srs=rsr$) we have an associated morphism between the corresponding Bott-Samelson bimodules  $$\mathrm{id^2}\otimes f_{sr}\otimes \mathrm{id} : B_pB_q (B_sB_rB_s) B_u\rightarrow B_pB_q(B_rB_sB_r) B_u.$$This means that for each  path $p$ in Rex($x$)  one can uniquely associate a morphism $f(p)$ between the corresponding Bott-Samelson bimodules. We call it  a \emph{path morphism.}

\begin{defi} A \emph{Complete path} in a graph is a path passing through every vertex of the graph at least once. 
\end{defi}

The following conjecture is due to the author.

\begin{ques}[Forking path conjecture] Let $x\in S_n.$  Let $p$ and $q$ be two complete paths in $\mathrm{Rex}(x)$. Then $f(p)=f(q).$
\end{ques}
Some remarks about this conjecture

\begin{remark} 
There has been some serious computer checking of this conjecture  by the author and Antonio Behn, using Geordie Williamson's programs.  That this conjecture was checked in huge cases (Rex graphs of over 100.000 vertices) was extremely surprising for the author. There is no  conceptual understanding of why this conjecture could be true. It seems utterly strange. 
\end{remark}

\begin{remark}
 If this conjecture was proved, then one would have a new and natural basis for the Hecke algebra, by applying the character map to the image of this (unique, for any element of $S_n$) projector. The author believes this would give a  rich combinatorial basis of the Hecke algebra, and a very natural set of Soergel bimodules that in our dreams should play a role in the understanding of the $p$-canonical basis\footnote{We will explain in a future paper of this saga what is this basis and why it is so important.}.
\end{remark}

\begin{remark}
The conjecture is not true for any Coxeter group. There are counter-examples for $F_4$ (although not easy to find). Nonetheless, in \cite{li2.5} the author proves an analogue of the Forking path conjecture for ``extra-large Coxeter groups'', i.e. groups in which $m(s,r)>3$ for all $s,r\in S, $ so it is quite mysterious for what groups this should be true. 
\end{remark}

\begin{remark}
There is a 55 pages paper of Elias \cite{El} published in a fine journal whose central result is  that a certain path morphism in Rex($w_0$) is an idempotent, where $w_0$ is the longest element of $S_n$, and that it projects to the indecomposable Soergel bimodule.  To prove that this path morphism is an idempotent has aproximately the same level of difficulty as to prove the Forking path conjecture for the particular case of the longest element $w_0$ and for  two specific complete paths $p_0$ and $q_0$. 
This gives an idea of how hard the Forking path conjecture can be, and also the kind of methods one can use to attack it (for example, the beautiful topology appearing in the higher Bruhat order of Manin and Schechtman).
\end{remark}

If one has a reduced expression $t_1\cdots t_i$ of some element $x\in W$, then by the important property in Section \ref{imp} we know that, if $xs<s$ for some $s\in S$, then  there is at least one path in Rex($x$) starting in $t_1\cdots t_i$ and ending in some reduced expression with $s$ in the far-most right position. One  considers the path morphism associated to a path like this and  call it \emph{a path morphism from $t_1\cdots t_i$ taking $s$ to the right. }

\subsection{The tree $\mathbb{T}_{\underline{s}}$}\label{T}

Let  $(W,\mathcal{S})$ be  an arbitrary Coxeter system. In this section we fix a sequence of simple reflections 
 $\underline{s}=(s_1,\ldots, s_n) \in \mathcal{S}^n$ and construct a tree $\mathbb{T}_{\underline{s}}$.

We construct a perfect binary tree (i.e a tree in which all interior nodes have exactly two children and all leaves have the same depth) with nodes colored by Bott-Samelson bimodules and arrows colored by morphisms from parent to  child nodes. We construct it by induction on the depth of the nodes. In depth zero and  one we have the following tree:

\vspace{0.5cm}

\centerline{
\xymatrix@C=0cm{
&(B_{s_1} B_{s_2} \cdots B_{s_n})\ar[ddl]_{m_{s_1}\otimes\mathrm{id}^{n-1}} \ar[ddr]^{ \mathrm{id}} & \\
&& \\
(B_{s_2} \cdots B_{s_n})&  & (B_{s_1})(B_{s_2} \cdots B_{s_n})\\
 }
}
\vspace{0.5cm}
Let $k<n$ and $\underline{t}=(t_1,  \cdots, t_{i})\in  \mathcal{S}^{i}$ be such that a node $N$ of depth $k-1$ is colored by the bimodule $(B_{t_1}  \cdots B_{t_{i}})(B_{s_{k}} \cdots B_{s_n}),$ then we have two cases.

\begin{enumerate}
\item If we have the inequality $l(t_1\cdots t_{i}s_k)>l(t_1\cdots t_{i})$, then the child nodes (of depth $k$)  and child edges of $N$ are colored in the following way:
\vspace{0.5cm}

 \centerline{
\xymatrix@C=0cm{
&(B_{t_1}  \cdots B_{t_{i}})(B_{s_{k}} \cdots B_{s_n})\ar[ddl]_{\mathrm{id}^{i}\otimes m_{s_{k}}\otimes\mathrm{id}^{}} \ar[ddr]^{ \mathrm{id}} & \\
&& \\
(B_{t_1}  \cdots B_{t_{i}})(B_{s_{k+1}} \cdots B_{s_n})&  & (B_{t_1}  \cdots B_{t_{i}}B_{s_{k}})(B_{s_{k+1}} \cdots B_{s_n})\\
 }
}
\item If we have the opposite inequality $l(t_1\cdots t_{i}s_k)<l(t_1\cdots t_{i})$, then the  child nodes (of depth $k$) and child edges of $N$ are colored in the following way (arrows are the composition of the corresponding pointed arrows):
 
 \vspace{0.5cm}
 \centerline{
\xymatrix@C=0cm{ 
& (B_{t_1}  \cdots B_{t_{i}})(B_{s_{k}} \cdots B_{s_n})\ar@/^15mm/[ddddr] \ar@/_15mm/[ddddl]
  \ar@{-->}[d]^{F\otimes \mathrm{id}}&&   \\
&B_{r_1}  \cdots B_{r_{i-1}}B_{s_{k}}B_{s_{k}} \cdots B_{s_n}\ar@{-->}[d]^{\mathrm{id}^{i-1}\otimes j_{s_k}\otimes \mathrm{id}}\\
&    B_{r_1}  \cdots B_{r_{i-1}}B_{s_{k}} \cdots B_{s_n}\ar@{-->}[ddl]_{\mathrm{id}^{i-1}\otimes m_{s_{k}}\otimes \mathrm{id}}
 \ar@{-->}[ddr]^{\mathrm{id}}&&  \\
 &&&  \\
(B_{r_1}  \cdots B_{r_{i-1}})(B_{s_{k+1}} \cdots B_{s_n})&& (B_{r_1}  \cdots B_{r_{i-1}}B_{s_{k}})(B_{s_{k+1}} \cdots B_{s_n})&
 & && &  & &  }
}

\end{enumerate}
The map $F$ is any path morphism from ${t_1}  \cdots {t_{i}}$ taking $s_k$ to the right (see Section \ref{f}). This finishes the construction of $\mathbb{T}_{\underline{s}}.$



By composing  the corresponding arrows we can see every leaf  of the tree $\mathbb{T}_{\underline{s}}$ colored by $B_{\underline{x}}$ as a morphism in the space $\mathrm{Hom}(B_{\underline{s}},B_{\underline{x}}).$
Consider the set $\mathbb{L}_{\underline{s}}(\mathrm{id}),$ the leaves of  $\mathbb{T}_{\underline{s}}$ that are colored by the bimodule $R$.  
 In \cite{Li1}  the set $\mathbb{L}_{\underline{s}}(\mathrm{id})$ is called \textit{light leaves basis} and the following theorem is proved.

\begin{thm}[Nicolas Libedinsky ]\label{LLB}
The set $\mathbb{L}_{\underline{s}}(\mathrm{id})$ is a basis of $\mathrm{Hom}(B_{\underline{s}},R)$ as a left $R$-module.
\end{thm}

In the paper \cite{Li1} the leaves colored with $R$ were called light leaves\footnote{The reason for this is that if the ``weight'' of a leaf is the length of the corresponding Bott-Samelson bimodule, then the leaves colored by $R$ are the ``lightest'' ones. }. Later  this notion was changed in the literature and any leaf in the tree $\mathbb{T}_{\underline{s}}$ got to be called a light leaf. We follow this new convention.

\subsection{Construction of the double leaves basis}\label{LL}

In this section we consider two arbitrary sequences (not necessarily reduced) of simple reflections $\underline{s}=s_1\cdots s_n$ and $\underline{r}=r_1\cdots r_p.$ We are interested in calculating the space  $\mathrm{Hom}(B_{\underline{s}}, B_{\underline{r}})$. 

\vspace{0.3cm}

\textbf{Philosophy: }\emph{The natural basis between Bott-Samelson bimodules is  the tree $\mathbb{T}_{\underline{s}}$   "pasted"  with  the  tree  $\mathbb{T}_{\underline{r}}$ inverted.} 

\vspace{0.3cm}

 For any light leaf $l:B_{\underline{r}}\rightarrow B_{\underline{t}}$ in $\mathbb{T}_{\underline{r}}$ we can find its \emph{adjoint light leaf} $l^a:B_{\underline{t}} \rightarrow  B_{\underline{r}}$ by replacing each morphism in the set $\{m_s, j_s, f_{sr}\}$ by its adjoint. So we obtain a tree $\mathbb{T}_{\underline{r}}^a$ where the arrows go from children to parents.

Let  $\mathbb{L}_{\underline{s}}$ be the set of light leaves of $\mathbb{T}_{\underline{s}}$ (recall that each leaf of the tree is seen as a morphism between Bott-Samelson bimodules).

Let  $h\in \mathbb{L}_{\underline{s}}$ and $g\in \mathbb{L}_{\underline{r}}^a$, with  $h\in\mathrm{Hom}(B_{\underline{s}},B_{\underline{x}})$ and $g \in\mathrm{Hom}(B_{\underline{y}},B_{\underline{r}})$, where $\underline{x}$ and $\underline{y}$ are reduced expressions of the elements $x, y\in W$ respectively. We define
$$g\cdot h=
\begin{cases}
\hspace*{0.1cm} g\circ F\circ  h \hspace*{0.3cm} \mathrm{if}\ x=y\\
\hspace*{0.7cm} \emptyset\hspace*{0.96cm} \mathrm{if}\ x\neq y
\end{cases}$$
where $F$ is any path morphism in $\mathrm{Hom}(B_{\underline{x}}, B_{\underline{y}})$.
 We call the set $\mathbb{L}_{\underline{r}}^a\cdot \mathbb{L}_{\underline{s}}$ the \emph{double leaves basis}  of $\mathrm{Hom}(B_{\underline{s}}, B_{\underline{r}}).$  The following theorem is proved in \cite{li3}.
 \begin{thm}[Nicolas Libedinsky ]\label{LL}
The Double Leaves Basis  is a basis as a  right (or left) $R$-module of the space $\mathrm{Hom}(B_{\underline{s}}, B_{\underline{r}}).$ 
\end{thm}
 
 The proof  of Theorem \ref{LL} is quite close to the proof of \cite[Th\'eor\`eme 5.1]{Li1}. 

\begin{remark}
Twice in this construction we said ``where $F$ is any path morphism''. Of course, for this basis to be well defined, these choices must be done once and for all. But what is quite striking about this theorem is that with any of these choices, Theorem  \ref{LL}  holds. There is a way to solve this ambiguity problem that we explore in the paper in preparation \cite{LW}. 
\end{remark}

\begin{ques} Give an algorithm to express a composition of double leaves as a linear combination of double leaves.
\end{ques}

\section{Final calculations: examples B, C and D}\label{final}

In section \ref{babyA} we were able to calculate the indecomposable Soergel bimodules for the baby example A. Now that we have introduced the graphical notation we are able to do the same thing (although we will not prove it) in the other examples, B, C and D. This method  gives  also a different way to calculate the indecomposable Soergel bimodules in baby example A. We will obtain them as the image of some idempotents of Bott-Samelson's instead of expressing them as tensor products of bimodules (which is not possible in general). 

It is enough to calculate example D (the Universal Coxeter group), since the formulas there will still be true for any Dihedral group. So in this section  we place ourselves in the case of the Universal Coxeter system $U_n$. Let us say that the box
 \begin{figure}[H] 
\begin{center}
 \includegraphics[scale=0.2]{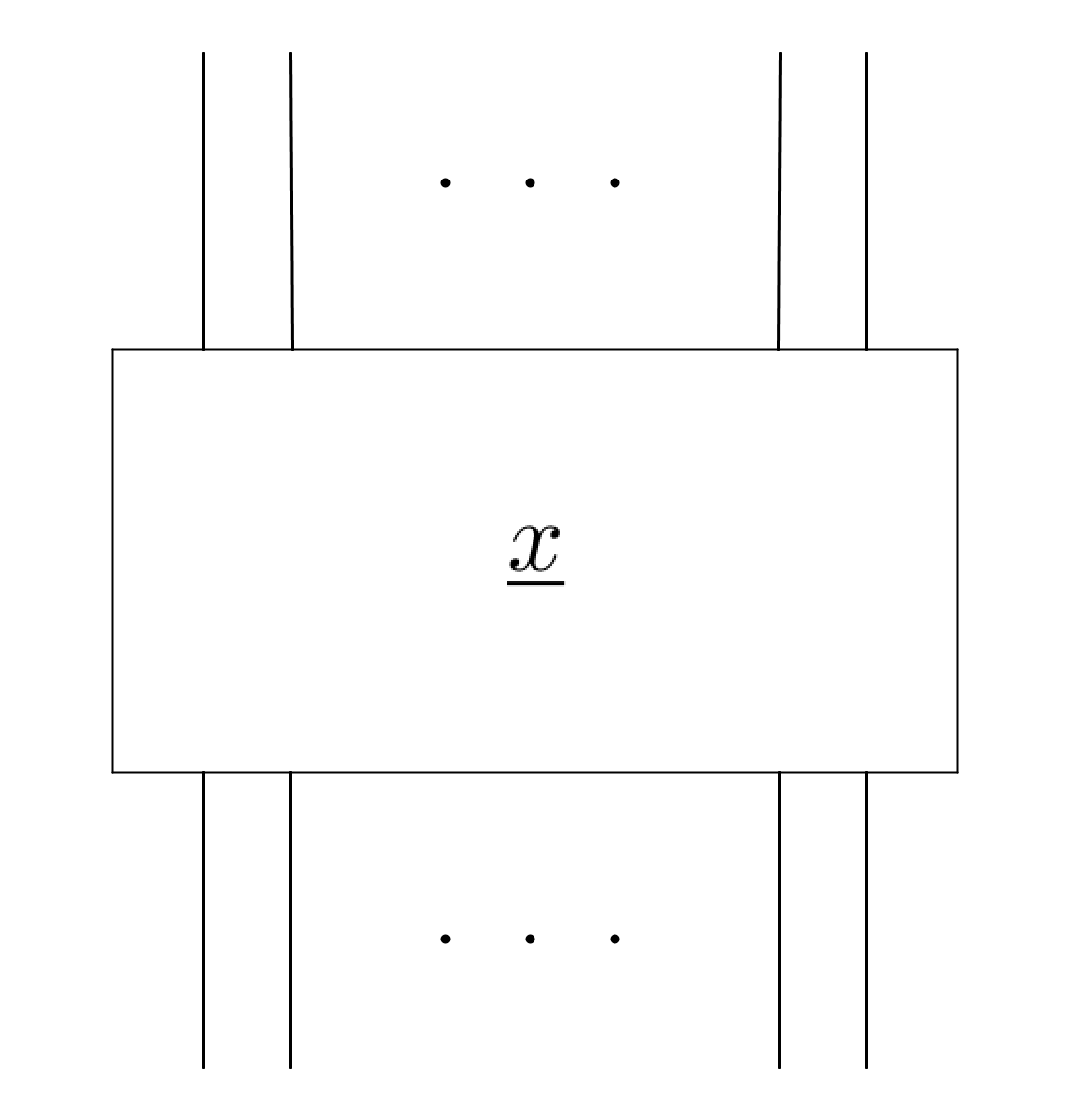}   
\end{center}
\caption{The projector}
\end{figure} 
expresses the projector (usually there are several projectors, but for this group one can prove that it is unique) into the indecomposable bimodule $B_x\subset B_{\underline{x}}$ where  $\underline{x}$ is a reduced expression in $U_n$.  For example 
\begin{figure}[H] 
\begin{center}
 \includegraphics[scale=0.75]{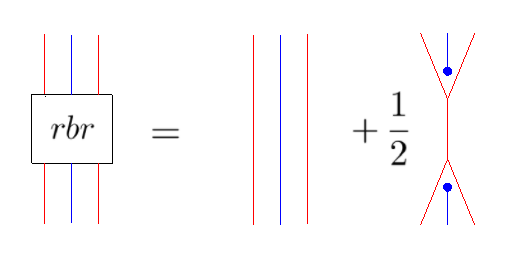} 
\end{center}
\end{figure} 
The following is a categorification of Dyer's Formula (Theorem \ref{Dy}) proved in \cite{EL}.

\begin{prop}[Ben Elias and Nicolas Libedinsky ] Let $x\in U_n$ and $\underline{x}=rs\cdots $ be a reduced expression of $x$ with $r, s, \ldots \in S.$  Let $r$ be represented by the color red, $t$ by the color tea green  and $b$ by the color blue. Black represents any color.  We have the following inductive formula. 

Extending $\underline{x}$ by a color $t\neq r, s$ one has 
 \begin{figure}[H] 
\begin{center}
 \includegraphics[scale=0.3]{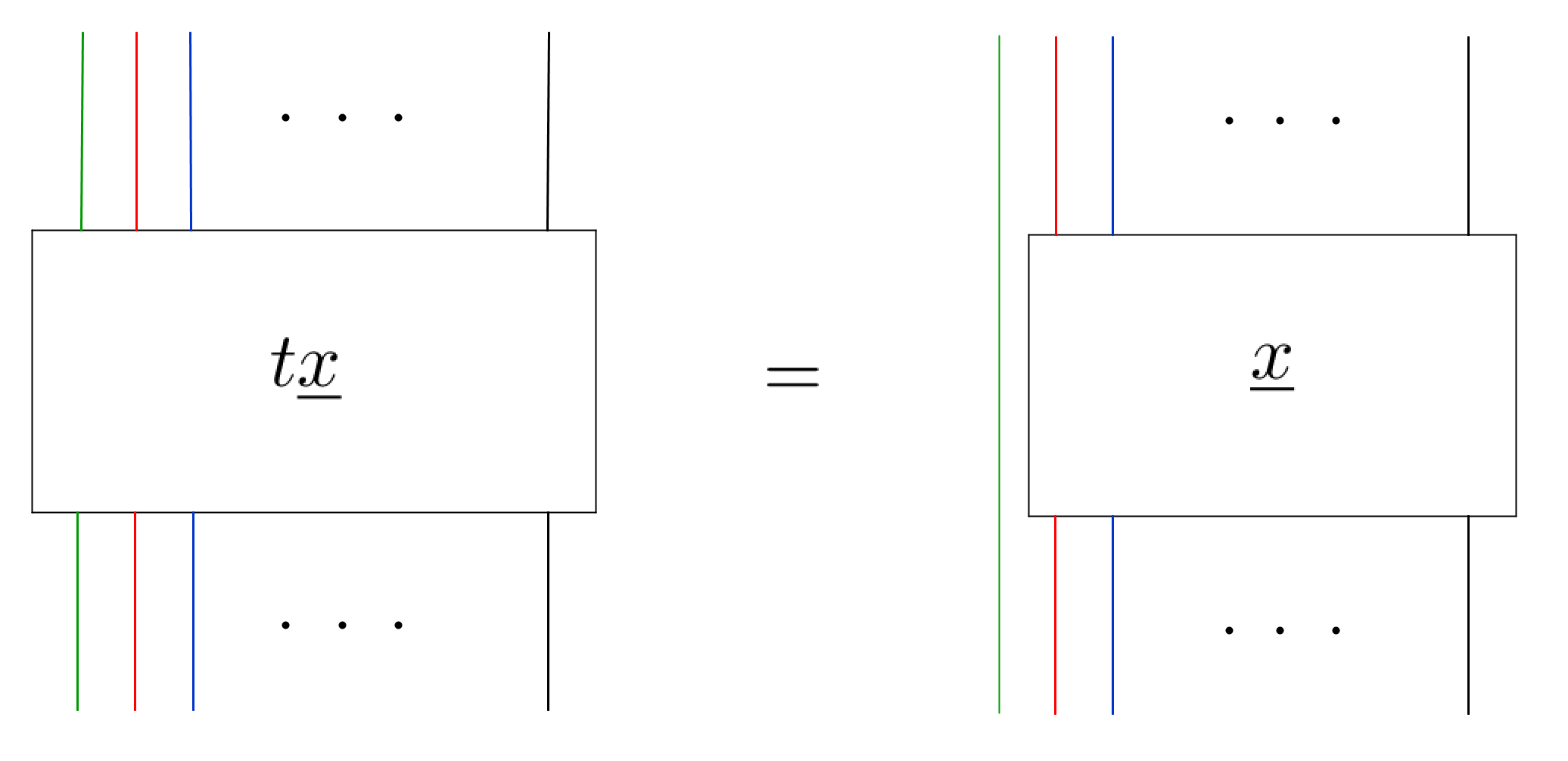}   
\end{center}
\end{figure} 
Extending $\underline{x}$ by $s$  one has 
 \begin{figure}[H] 
\begin{center}
 \includegraphics[scale=0.4]{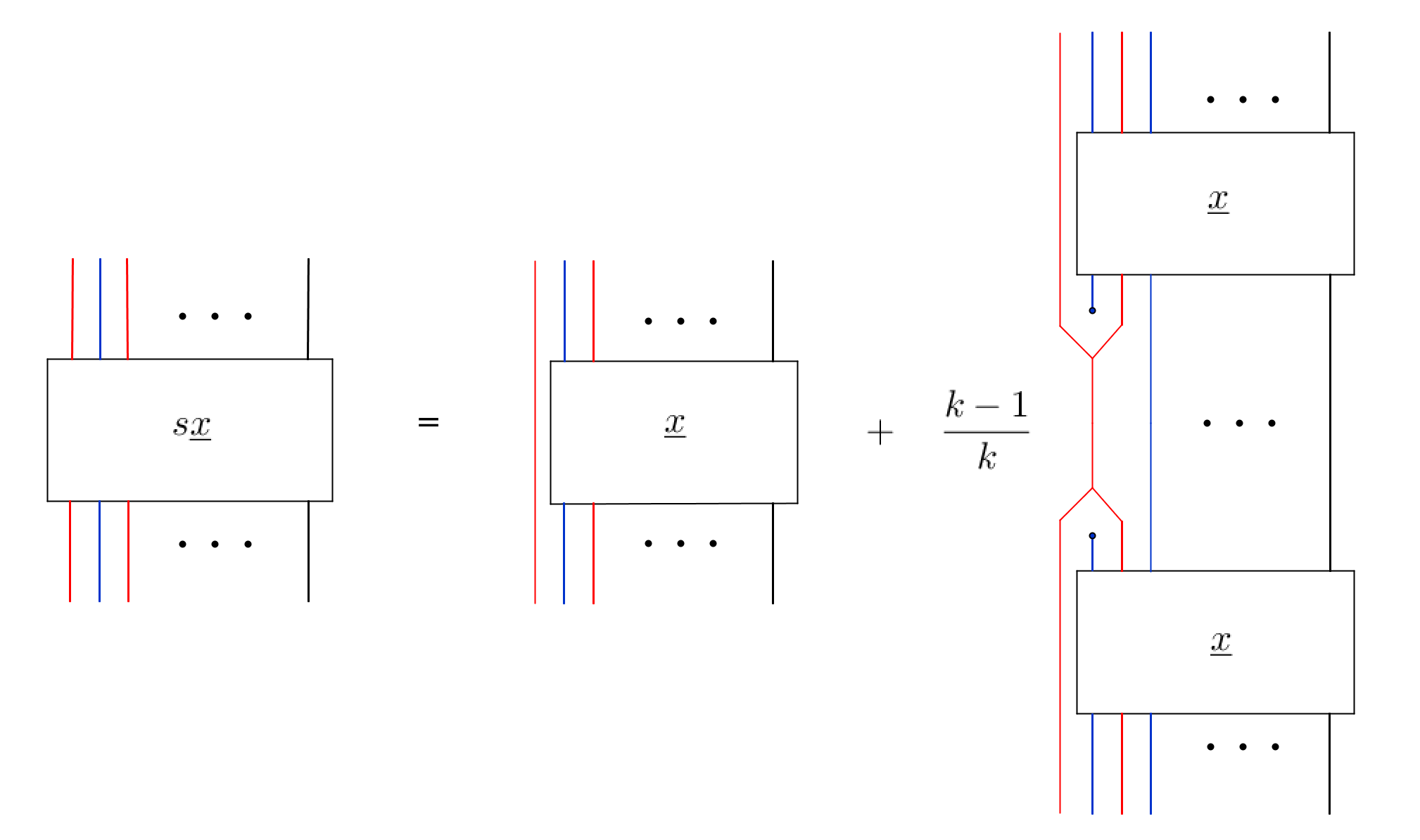}   
\end{center}
\end{figure} 
\end{prop}

To convince the reader that this is a categorification of Dyer's formula, one should remark that the rightmost term, when multiplied by $-1$, is an idempotent of $\mathrm{End}(s\underline{x})$ projecting to $B_{rx}$.

\end{document}